\documentclass[12pt,leqno]{article}
\usepackage{amsfonts}
\usepackage{amsmath, amsthm, amsfonts, amssymb, color}
\usepackage{mathrsfs}
\usepackage{authblk}
\usepackage{color}
\newtheorem{thm}{Theorem}[section]
\newtheorem{cor}[thm]{Corollary}

\theoremstyle{definition}

\newcommand{\scr}[1]{\mathscr #1}
\definecolor{wco}{rgb}{0.5,0.2,0.3}

\numberwithin{equation}{section} \theoremstyle{remark}

\newcommand{\ua}{\uparrow}

\def\l{\left} \def\r{\right}

\title{{\bf  Probability Versions of Li-Yau Type Inequalities and Applications  }\footnote{Supported in
 part by   the National Key R\&D Program of China (No. 2022YFA1006000, 2020YFA0712900) and NNSFC (11921001).} }

 \author[1]{Li-Juan Cheng}  \author[2]{Feng-Yu Wang}

\affil[1]{\small School of Mathematics, Hangzhou Normal
  University,\par
  Hangzhou 311121, People's Republic of China\par
  \texttt{lijuan.cheng@hznu.edu.cn}\vspace{1em}}

   \affil[2]{\small Center for Applied Mathematics, Tianjin
      University,\par Tianjin 300072, People's Republic of China \par
      \texttt{wangfy@tju.edu.cn}}

\begin{document}
\allowdisplaybreaks

\def\R{\mathbb R}  \def\ff{\frac} \def\ss{\sqrt} \def\B{\mathbf
B}
\def\N{\mathbb N} \def\kk{\kappa} \def\m{{\bf m}}
\def\ee{\varepsilon}\def\ddd{D^*}
\def\dd{\delta} \def\DD{\Delta} \def\vv{\varepsilon} \def\rr{\rho}
\def\<{\langle} \def\>{\rangle}
  \def\nn{\nabla} \def\pp{\partial} \def\E{\mathbb E}
\def\d{\text{\rm{d}}} \def\bb{\beta} \def\aa{\alpha} \def\D{\scr D}
  \def\si{\sigma} \def\ess{\text{\rm{ess}}}\def\s{{\bf s}}
\def\beg{\begin} \def\beq{\begin{equation}}  \def\F{\scr F}
\def\Ric{\mathcal Ric} \def\Hess{\text{\rm{Hess}}}
\def\e{\text{\rm{e}}} \def\ua{\underline a} \def\OO{\Omega}  \def\oo{\omega}
 \def\tt{\tilde}\def\[{\lfloor} \def\]{\rfloor}
\def\cut{\text{\rm{cut}}} \def\P{\mathbb P} \def\ifn{I_n(f^{\bigotimes n})}
\def\C{\scr C}      \def\aaa{\mathbf{r}}
\def\gap{\text{\rm{gap}}} \def\prr{\pi_{{\bf m},\nu}}
\def\Z{\mathbb Z} \def\vrr{\nu} \def\ll{\lambda}
\def\L{\scr L}\def\Tt{\tt} \def\TT{\tt}\def\II{\mathbb I}
\def\i{{\rm in}}\def\Sect{{\rm Sect}}  \def\H{\mathbb H}
\def\M{\mathbb M}\def\Q{\mathbb Q} \def\texto{\text{o}} \def\LL{\Lambda}
\def\Rank{{\rm Rank}} \def\B{\scr B} \def\i{{\rm i}} \def\HR{\hat{\R}^d}
\def\to{\rightarrow} \def\gg{\gamma}
\def\EE{\scr E} \def\W{\mathbb W}
\def\A{\scr A} \def\Lip{{\rm Lip}}\def\S{\mathbb S}
\def\BB{\scr B}\def\Ent{{\rm Ent}} \def\i{{\rm i}}\def\itparallel{{\it\parallel}}
\def\g{{\mathbf g}}\def\Sect{{\mathcal Sec}}\def\T{\mathcal T}\def\BB{{\bf B}}
\def\f\ell \def\g{\mathbf g}\def\BL{{\bf L}}  \def\BG{{\mathbb G}}
\def\Bd{{D^E}} \def\BdP{D^E_\phi} \def\Bdd{{\bf \dd}} \def\Bs{{\bf s}} \def\GA{\scr A}
\def\Bg{{\bf g}}  \def\Bdd{\psi_B} \def\supp{{\rm supp}}\def\div{{\rm div}}
\def\ddiv{{\rm div}}\def\osc{{\bf osc}}\def\1{{\bf 1}}\def\BD{\mathbb D}
\def\H{{\bf H}}\def\gg{\gamma} \def\n{{\mathbf n}}\def\GG{\Gamma}\def\HAT{\hat}
\def\SU{{\bf SU}}\def\Ric{{\rm Ric} }
\maketitle

\begin{abstract}  By using stochastic analysis, two  probability versions of Li-Yau  type inequalities are established for diffusion semigroups on a manifold possibly with   (non-convex) boundary. The inequalities are explicitly given by  the Bakry-Emery  curvature-dimension,  as well as the lower bound   of the second fundamental form if the boundary exists.
As applications,   a number of global and local   estimates are presented, which  extend  or improve existing ones derived   for manifolds without boundary.
Compared with the maximum principle technique developed in the literature, the probabilistic argument we used is more straightforward and hence considerably simpler.
\end{abstract} \noindent

 \vskip 2cm

 \section{Introduction}
Since Li-Yau  \cite{Li-Yau} established their famous parabolic Harnack inequality for the heat semigroup  on Riemannian manifolds, 
 a number of  extensions and refinements   have been intensively made in the literature, which   will be briefly recalled latter on.

The purpose of this paper is to provide probability versions of Li-Yau type inequalities for the  diffusion semigroups   on a complete Riemannian manifold  possibly with a Neumann boundary, 
which are formulated by  expectations on   functionals of the corresponding diffusion process, which are  explicitly given by the Bakry-Emery curvature-dimension of the generator,   the second fundamental form of the boundary if exists, and   an adapted process $\ell_s$ satisfying $\ell_0=1$ and $\ell_t=1$,  see Theorem \ref{T1}. With specific  choices of the reference adapted process $\ell_s$, these inequalities imply new explicit gradient estimates on the heat semigroup, see Corollary \ref{C1} for global estimates and  Corollary \ref{C2} for local estimates.
 
 Compared with the maximum principle technique developed from \cite{Li-Yau} and adopted in   substantial  references,   the martingale argument we used here considerably simplify the proof. The main idea of the study comes from   Arnaudon-Thalmaier \cite{AT}, where some global and local gradient estimates on the heat semigroup is presented by using stochastic analysis on manifolds.

Before moving on, let us recall some existing results on Li-Yau type inequalities, which are   derived on manifolds without boundary, and in most cases for the Laplacian without drift.
See \cite{W1,W2} for extensions to manifolds with boundary.

Let $M$ be an $m$-dimensional connected complete Riemannian manifold without boundary,   let $L:=\DD+Z $ for some vector field  $ Z.$  Assume that
for some constants $n\ge m$ and   $K\in \R$ the following  Bakry-Emery curvature-dimension condition holds:
\beq\label{C}  \Ric_{Z}^{(n-m)}:=\Ric_{Z}-\ff{Z\otimes Z} {n-m}\ge K,\end{equation}
where $\Ric_Z:= \Ric-\nn Z$ is the Bakry-Emery curvature, and $\Ric$ is the Ricci curvature. This condition
  is equivalent to
\begin{align}\label{CD-L}
\frac{1}{2}L|\nabla f|^2-\langle \nabla Lf, \nabla f \rangle\ge K|\nabla f|^2+\frac{1}{n}(Lf)^2,\quad f\in C^{\infty}(M).
\end{align}
When $Z=0$ we may take $n=m$ in \eqref{C}, so that this condition reduces to $\Ric\ge K.$

 Consider a bounded positive solution to the heat equation
$$\pp_t u_t= L u_t,\ \ t\ge 0.$$
Li-Yau  \cite{Li-Yau}  proved   that  when  $Z=0$ and $\Ric\ge K$ for some constant $K$,
$$
\frac{|\nabla u_t|^2}{u_t^2}\le \alpha \frac{\DD u_t}{u_t}+ \frac{n K^-\alpha^2 }{2(\alpha-1)}+\frac{n\alpha^2}{2t},\ \ t>0,\aa>1,
$$ where $K^-:= \max\{-K,0\}$ is the negative part of $K$.
In particular, when $K=0$ (i.e. $\Ric\ge 0$) with $\aa\downarrow 1$,  this implies
$$\frac{|\nabla u_t|^2}{u_t^2}\le   \frac{\DD u_t}{u_t}+  \frac{n}{2t},$$
where the equality holds for $u_t$ being the standard heat kernel   on $M=\R^n$.

The above Li-Yau  inequality has been extensively extended or refined.
For instances,   by  Davies  \cite{Davies} (for $Z=0$)
\begin{align}\label{LY1}
\frac{|\nabla u_t|^2}{u_t^2}\le \alpha \frac{\DD u_t}{u_t}+\frac{nK^-\alpha ^2 }{4(\alpha -1)}+\frac{n\alpha^2}{2t},\ \ \aa>1, t>0;
\end{align}
  by  Yau \cite{Yau95} (for $Z=0$)
$$\frac{|\nabla u_t|^2}{u_t^2}\le   \frac{\DD u_t}{u_t}+ \sqrt{2nK^-}\sqrt{\frac{|\nabla u_t|^2}{u_t}+\frac{n}{2t}+2nK^-}+\frac{n}{2t},\ \ t>0,$$
which is then improved      by Bakry-Qian  \cite[(6)]{B-Q}
 $$\frac{|\nabla u_t|^2}{u_t^2}\le   \frac{\DD u_t}{u_t}+ \sqrt{nK^-} \sqrt{\frac{|\nabla u_t|^2}{u_t^2}+\frac{n}{2t}+\frac{nK^-}{4}}+\frac{n}{2t},$$ by    \cite[(54)]{B-Q}
\beq\label{LY3} \frac{|\nabla u_t|^2}{u_t^2}\le \l(1+\frac{2}{3}K^- t\r)\frac{\DD u_t}{u_t}+ \frac{n}{2t}+\frac{nK^-}{2}\l(1+\frac{1}{3}K^-t\r),\ \ t>0;
 \end{equation}
and more recently by Bakry-Bolley-Gentil \cite{BBG}  (also for $Z\ne 0$)
\beq\label{J'} \ff 4 {nK}\ff {Lu_t}{u_t}< 1+\ff{\pi^2}{K^2t^2},\end{equation}
\beq\label{BBG}
 \frac{|\nabla u_t|^2}{u_t^2}\le   \frac{L u_t}{u_t}- \frac{nK }{2}+ \frac{n}{2}\Phi_t\Big(1-\frac{4}{nK}\frac{L u_t}{u_t}\Big),\ \ t>0,\end{equation}
 where
$$\Phi_t(r):=\begin{cases}
K\sqrt{r}\coth(Kt\sqrt{r}),& \ \ \  \  r> 0,\\
\ff 1 t, & \ \ \  \  r= 0,\\
K\sqrt{-r}\cot(Kt\sqrt{-r}), &\  \  \  \   -\frac{\pi^2}{K^2 t^2}<r< 0.
\end{cases}
$$
Moreover,
Li-Xu  \cite{Li-Xu} proved (for $Z=0$)
\begin{align}\label{L-X-1}
\frac{|\nabla u_t|^2}{u_t^2}\le \l(1+\frac{\sinh(K^-t)\cosh(K^-t)-K^-t}{\sinh^2(K^-t)}\r)\frac{\DD u_t}{u_t}+ \frac{nK^-}{2}\Big(1+\coth(K^-t)\Big),\ \ t>0,
\end{align} see also Qian \cite{Qian14, Qian24} for conditions on functions
$a$ and $ c $ such that
\begin{align*}
\frac{|\nabla u_t|^2}{u_t^2}\le a(t)\frac{\DD u_t}{u_t}+ c(t),\ \ t>0.
\end{align*}
 All these inequalities are proved by using the technique of maximum principle developed in Li-Yau \cite{Li-Yau}.

  \

In the next section, we present two  probability versions of Li-Yau type inequalities for the  heat semigroup generated by $L:=\DD+Z $ with Neumann boundary if $\pp M$ exists, by using the  diffusion process $X_t$ generated by $L$ with reflecting boundary if $\pp M$ exists. The inequalities are explicitly given by $K$ and $n$ in \eqref{C}   for some constant $n\ge m$ and a function $K\in C(M)$, and the lower bound of the second fundamental form  of $\pp M$ if it exists.  As applications, besides extensions of existing estimates to the case with boundary,  some new global and local estimates are presented in Sections 3 and 4 respectively, where the curvature may be unbounded from below.

 \section{General results }

Let $M$ be an $m$-dimensional connected complete Riemannian manifold possibly with a boundary $\pp M$.
Let $L:=\DD+Z $ for some vector field  $ Z $  such that \eqref{C} holds for some constant $n\ge m$ and a function $K\in C(M)$.
When $\pp M$ exists, let $\si\in C(\pp M) $ be a lower bound of  the second fundamental form of $\pp M$,   i.e. the inward unit normal vector field $N$ of $\pp M$ satisfies
\beq\label{S} \II(v,v):= -\<\nn_v N,v\>\ge \si |v|^2,\ \ v\in T\pp M.\end{equation}

From now on, let $0\le u_0\in \D(L) \cap C_b^2(M)$ be positive with bounded $  u_0 +|Lu_0|$, and  $Nu_0|_{\pp M}=0$ if $\pp M$ exists. Let  $u: [0,\infty)\times M\to (0,\infty)$  solve
the following heat equation
\beq\label{H-equation} \pp_t u_t(x) = L u_t(x),\ \ Nu_t|_{\pp M}=0,\ \ t\ge 0, \ x\in M,\ \end{equation} where the Neumann boundary condition $Nu_t|_{\pp M}=0$   applies only when $\pp M$ exists.

Let $X_t$ be the  diffusion process generated by $L$ with reflecting boundary if $\pp M$ exists, which can be constructed as the unique solution to the following SDE on $M$:
\beq\label{SDE} \d X_t= Z(X_t)\d t+ \ss 2\, U_t\circ \d B_t+ N(X_t)\d\scr L_t,\end{equation}
where $B_t$ is the standard $m$-dimensional Brownian motion on a probability space $(\OO,\F,\P)$ with natural filtration $\{\F_t\}_{t\ge 0}$, $U_t$ is the horizontal lift
of $X_t$ to the frame bundle $O(M),$ and $\scr L_t$ is the local time of $X_t$ on $\pp M$ if exists, so that  $\scr L_t=0$ if $\pp M$ does not exist.
For any $x\in M$ let $\E^x$ denote the expectation taken for the diffusion   process  with initial value $X_0=x$. By \eqref{H-equation} and It\^o's formula, we have
\beq\label{XW} u_t(x)=\E^x[u_0(X_t)],\ \ Lu_t(x)= \E^x[L u_0(X_t)],\ \ t\ge 0, \,  x\in M.\end{equation}
In the following two subsections, we present a global probability version and a local probability version
of Li-Yau type inequalities for $u_t$ respectively.

 When $\pp M$ is either empty or convex (i.e. $\si=0$) so that $\si(X_s)\d\scr L_s=0$, and $K$ is a constant, by choosing deterministic $\ell_s$ with $\ell_0=1$ and $\ell_t=0$, and applying the equations in \eqref{XW},  the following estimate \eqref{J0} reduces to
\beq\label{WX0} \frac{|\nabla u_t|^2}{u_t^2} \le \ff n 2 \int_0^t |\ell_s'|^2 \e^{-2Ks}\d s  -   \ff{Lu_t}{u_t}  \int_0^t(\ell_s^2)'  \e^{-2Ks}  \, \d s,\ \ t>0,\end{equation}
which  has been proved  in \cite[Proposition 2.4]{BG11} (see also  \cite{Qian14})  by using   analytic  arguments.

\beg{thm}  \label{T1}
Assume      $\eqref{C}$  for some constant $n\ge m$ and a function $K\in C(M)$, and also $\eqref{S}$    for some   $\si\in C(\pp M)$ if $\pp M$ exists. Let $t>0, x\in M,$ and $(\ell_s)_{s\in [0,t]}$ be an  adapted real process such that $\ell_0=1,\ell_t=0$,   $\ell_s'$  exists   $\d s\times \P$-a.e. on $ [0,t]\times \OO,$ and
\beq\label{CV0'} \beg{split}& \E^x\bigg[ \sup_{s\in [0,t]}  \l(\ell_s^2 \e^{\int_0^s[K (X_r)^-\, \d r+\si(X_r)^-\, \d \scr L_r]} \frac{|\nn u_{t-s}|^2}{u_{t-s}}(X_s) \r)\bigg]\\
&\quad + \E^x\left[\int_0^t\big(|\ell'_s|^2 +\ell_s^2\big)\e^{-2\int_0^s[K (X_r)\, \d r+\si(X_r)\, \d \scr L_r]}\, \d s\right]<\infty.\end{split}
\end{equation}
Then
\beg{enumerate} \item[$(1)$]
\beq\label{J0} \beg{split}
\frac{|\nabla u_t|^2}{u_t}(x)
&\leq \frac{n}{2} \,\E^x\left[u_0(X_t)\int_0^{t} |\ell_s'|^2 \e^{-2\int_0^s \{K (X_r)\,\d r+\si(X_r)\,\d\scr L_r\}}\, \d s\right]  \\
&\quad -   \E^x \left[ (Lu_0)(X_t) \int_0^{t}(\ell_s^2)'  \e^{-2\int_0^s \{K (X_r)\,\d r+ \si(X_r)\d\,\scr L_r\}}  \, \d s \right];
\end{split}\end{equation}
\item[$(2)$]  Moreover, if $\ell_s$ is deterministic with $\ell_s'\le 0$ and $\sigma= 0$ (i.e. $\pp M$ is convex or empty), then
  for any $\aa\in (1,\infty)$ and any constant $K_0$ such that
 $K\geq K_0$,
\beq\label{A1} \begin{split}
& (1+\gamma_{t,\alpha}) \ff{|\nabla u_t|^2}{u_t} (x) -    L u_t (x) \\
& \le  \frac{n\alpha  }{2 } \,  \E^x \bigg[u_0(X_t)\int_0^{t}\Big(\frac{K (X_s)}{\alpha-1}\ell_s+\ell'_s\Big)^2\e^{\frac{2}{\alpha-1}\int_0^s K(X_r)\,\d r   }\,\d s\bigg],\end{split}
\end{equation}
where $$\gg_{t,\aa}   :=  2 K_0 \int_0^{t}  \ell_s^2 \e^{\frac{2\aa}{\alpha-1}K_0 s }  \,\d s>\ff 1\aa -1.$$
 \end{enumerate} \end{thm}

\begin{proof}  (a)  When $\pp M$ exists, noting that
   $N u_s|_{\pp M}=0$ implies  $\nn u_s|_{\pp M}\in T\pp M$, we derive from \eqref{S} that on $\pp M$,
\beq\label{N} N \ff{|\nn u_s|^2}{u_s} =  \ff{N |\nn u_s|^2 }{u_s} =  \ff{2\II(\nn u_s,\nn u_s)}{u_s} \ge 2\si(X_s) \ff{|\nn u_s|^2}{u_s},\ \ s>0.\end{equation}
Next,  $u$ is the solution to the equation  \eqref{H-equation}   implies
\beq\label{N2} (L +\pp_s) u_{t-s}= 0= (L +\pp_s) L u_{t-s},\ \ \ s\in [0,t).\end{equation}
This together  with the Bochner-Weizenb\"ock formula leads to
\beg{align*} &(L+\pp_s)\bigg(\ff{|\nn u_{t-s}|^2}{u_{t-s} }\bigg)= \ff1{u_{t-s}}\big(L |\nn u_{t-s}|^2 -2\<\nn L u_{t-s},\nn u_{t-s}\>\big) \\
&\qquad  -  \ff{4}{u_{t-s}^2}\Hess_{u_{t-s}}( \nn u_{t-s},\nn u_{t-s})
+\ff{2|\nn u_{t-s}|^4}{u_{t-s}^3}\\
&=\ff 2 {u_{t-s}}\bigg( \|\Hess_{u_{t-s}}\|_{HS}^2 -\ff 2 {u_{t-s}} \Hess_{u_{t-s}}(\nn u_{t-s}, \nn u_{t-s})+ \ff{|\nn u_{t-s}|^4}{u_{t-s}^2}  + \Ric_Z(\nn u_{t-s},\nn u_{t-s}) \bigg) \\
&=\ff 2 {u_{t-s}} \bigg\|\Hess_{u_{t-s}} -\ff{\nn u_{t-s}\otimes\nn u_{t-s}}{u_{t-s}}\bigg\|_{HS}^2  + \ff 2 {u_{t-s}} \Ric_Z(\nn u_{t-s},\nn u_{t-s})\\
&\ge \ff 2 {m  u_{t-s} } \bigg(\DD u_{t-s}-\ff{|\nn u_{t-s}|^2}{u_{t-s}} \bigg)^2 + \ff 2 {u_{t-s}} \Ric_Z(\nn u_{t-s},\nn u_{t-s}).\end{align*}
Combining this with \eqref{C} and the fact that
\beg{align*}&\ff 1 {m} \bigg(\DD u_{t-s}-\ff{|\nn u_{t-s}|^2}{u_{t-s}} \bigg)^2 = \ff 1 m \bigg(L u_{t-s}-\ff{|\nn u_{t-s}|^2}{u_{t-s}} -Z u_{t-s}\bigg)^2\\
&\ge \ff 1 n \bigg(L u_{t-s}-\ff{|\nn u_{t-s}|^2}{u_{t-s}} \bigg)^2
-\ff{|Z u_{t-s}|^2}{n-m},\end{align*}
we derive
\beq\label{P1}  (L+\pp_s)\bigg(\ff{|\nn u_{t-s}|^2}{u_{t-s} } \bigg) \ge \ff{2 }{n u_{t-s}} \bigg(L u_{t -s}-\ff{|\nn u_{t-s}|^2}{u_{t-s}} \bigg)^2  + 2K\ff{|\nn u_{t-s}|^2}{u_{t-s}},\
\ s\in [0,t].  \end{equation}

 In the following, we use the above estimates and It\^o's formula to prove \eqref{J0} and \eqref{A1} respectively.  For simplicity,
   let $\overset{m}{\ge }$, $\overset{m}{\le }$ and    $\overset{m}{=}$  denote the corresponding inequalities and  equality up to an additive  local martingale term.

(b) Let $h_s=\ell_s^2$.  By \eqref{SDE}, \eqref{N}, \eqref{P1} and  It\^{o}'s formula,   we obtain
\begin{align*}
&\d\left(h_s \e^{ -2\int_0^s [K (X_r)\,\d r +\si(X_r)\,\d\scr L_r] }  \frac{|\nabla u_{t-s}|^2}{u_{t-s}} (X_s)\right)  \\
& \overset{m}{=}\e^{ -2\int_0^s [K (X_r)\,\d r +\si(X_r)\,\d\scr L_r]}\bigg\{\big( h_s'  -2h_s K(X_s)\big)  \frac{|\nabla u_{t-s}|^2}{u_{t-s}} (X_s) + h_s\Big[\big(L+\pp_s\big) \frac{|\nabla u_{t-s}|^2}{u_{t-s}} \Big](X_s) \bigg\}\d s\\
&\qquad +h_s \e^{ -2\int_0^s [K (X_r)\,\d r +\si(X_r)\,\d\scr L_r]}\bigg(N \l( \frac{|\nabla u_{t-s}|^2}{u_{t-s}} \r) -2\si  \frac{|\nabla u_{t-s}|^2}{u_{t-s}}\bigg)(X_s)\,\d\scr L_s\\
&  \ge   h_s'   \e^{ -2\int_0^s[K (X_r)\,\d r +\si(X_r)\,\d\scr L_r] }  \bigg(\frac{|\nabla u_{t-s}|^2}{u_{t-s}} -L u_{t-s} +Lu_{t-s}\bigg)(X_s)\, \d s\\
&\qquad + \ff{2 h_s}{nu_{t-s}(X_s)} \e^{ -2\int_0^s[K (X_r)\,\d r +\si(X_r)\,\d\scr L_r]} \bigg(L u_{t -s}-\ff{|\nn u_{t-s}|^2}{u_{t-s}} \bigg)^2(X_s)\, \d s\\
&\ge    -\frac{n|h'_s|^2}{8h_s}\e^{-2\int_0^s[ K (X_r)\,\d r +2 \si(X_r)\,\d\scr L_r]}u_{t-s}(X_s)\, \d s +  h'_s
\e^{-2\int_0^s[K (X_r)\,\d r +\si(X_r)\,\d\scr L_r]} L u_{t-s}(X_s)\, \d s.
\end{align*}
Since $h_s=\ell_s^2, h_0= 1$ and $h_t =0$, and $\|u\|+\|Lu\|$ is bounded on $[0,t]\times M$, by \eqref{CV0'} and the dominated convergence theorem, this implies
\beq\label{*} \beg{split}  \ff{|\nn u_t|^2}{u_t}(x)&\le \ff n 2\, \E^x \bigg[\int_0^t u_{t-s}(X_s)  |\ell_s'|^2 \e^{-2\int_0^s [ K (X_r)\,\d r + \si(X_r)\,\d\scr L_r]}\,\d s\bigg]\\
 &\quad -  \E^x\bigg[\int_0^t (L u_{t-s})(X_s)   (\ell_s^2)'   \e^{-2\int_0^s [ K (X_r)\,\d r + \si(X_r)\,\d\scr L_r]}\,\d s\bigg].\end{split}\end{equation}
Noting that \eqref{XW} and the Markov property  imply
\begin{align}\label{MM1} (L u_{t-s})(X_s)=\E^{x} (Lu_0(X_t)|\F_s),\ \ \ u_{t-s}(X_s)= \E^{x} (u_0(X_t)|\F_s),\end{align}
we derive \eqref{J0}.

(c)  Let $\si=0$ and $\ell_s$ be deterministic with $\ell_s'\le 0$.  Noting that $NLu_{t-s} |_{\pp M}=0$ for $s\in [0,t)$,
 by \eqref{SDE}, \eqref{N} for $\si=0$,  and   It\^o's formula, we obtain
\beg{align}\label{Ito-M}&\d\bigg(\e^{\frac{2}{\alpha-1}\int_0^s K(X_r)\,\d r}\ell_s^2 \Big( \frac{|\nabla u_{t-s}|^2}{u_{t-s}}-\alpha L u_{t-s}\Big)(X_s)\bigg)\notag\\
&\overset{m}{\ge}\e^{\frac{2}{\alpha-1}\int_0^s K(X_r)\,\d r} \Big(\frac{2K(X_s)}{\alpha-1}\ell_s^2+2\ell_s \ell'_s \Big) \l( \frac{|\nabla u_{t-s}|^2}{u_{t-s}} -\alpha Lu_{t-s}\r)(X_s)\, \d s\notag\\
&\quad  + \e^{\frac{2}{\alpha-1}\int_0^s K(X_r)\,\d r}\ell_s^2  \bigg[\big(L+\pp_s\big)  \bigg( \frac{|\nabla u_{t-s}|^2}{u_{t-s}}-\alpha L u_{t-s}\bigg)\bigg](X_s)\, \d s,\ \ s\in [0,t].\end{align}
Combining this with \eqref{N2} and   \eqref{P1},  we obtain
  \begin{align*}
&\d\bigg(\e^{\frac{2}{\alpha-1}\int_0^s K(X_r)\,\d r}\ell_s^2 \Big( \frac{|\nabla u_{t-s}|^2}{u_{t-s}}-\alpha L u_{t-s}\Big)(X_s)\bigg) \\
&\overset{m}{\ge} 2 \alpha \e^{\frac{2}{\alpha-1}\int_0^s K(X_r)\,\d r} \Big(\frac{K(X_s)}{\alpha-1}\ell_s^2+ \ell'_s \ell_s \Big) \Big(\frac{|\nabla u_{t-s}|^2}{u_{t-s}}-L u_{t-s}\Big) (X_s)\, \d s\\
&\quad  +\frac{2\e^{\frac{2}{\alpha-1}\int_0^s K(X_r)\,\d r}\ell_s^2}{nu_{t-s}(X_s)}   \Big(L \,  u_{t-s}-\frac{|\nabla u_{t-s}|^2}{u_{t-s}} \Big)^2(X_s)\, \d s \\
 &\quad  -2(\alpha-1) \e^{\frac{2}{\alpha-1}\int_0^s K(X_r)\,\d r} \ell_s \ell'_s \frac{|\nabla u_{t-s}|^2}{u_{t-s}} (X_s)\, \d s\\
 &\ge -\ff{n\aa^2}2  \e^{\frac{2}{\alpha-1}\int_0^s K(X_r)\,\d r}  \Big(\frac{K(X_s)}{\alpha-1}\ell_s+ \ell'_s  \Big)^2 u_{t-s}(X_s)\, \d s\\
 &\quad - 2(\alpha-1) \e^{\frac{2}{\alpha-1}\int_0^s K(X_r)\,\d r} \ell_s \ell'_s \frac{|\nabla u_{t-s}|^2}{u_{t-s}} (X_s)\, \d s,\ \ s\in [0,t].
\end{align*}
 Combining the condition \eqref{CV0'} with the boundedness of $u$ and $Lu$, $\ell_0=1$ and $\ell_t=0$, this  implies
\beq\label{P2-0}  \ff{|\nabla u_t|^2}{u_t}(x)  - \alpha L u_t (x) \le  \ff{n\aa^2}2 I_1+ (\alpha-1) I_2,
\end{equation}
where by \eqref{MM1},
 \beg{align*}
 I_1&:=   \E^x\l[ \int_0^t \e^{\frac{2}{\alpha-1}\int_0^s K(X_r)\,\d r} \Big(\frac{K(X_s)}{\alpha-1}\ell_s+ \ell'_s \Big)^2 u_{t-s}(X_s)\, \d s\r]\\
&=  \E^x\bigg[u_0(X_t) \int_0^t  \e^{\frac{2}{\alpha-1}\int_0^s K(X_r)\,\d r} \Big(\frac{K(X_s)}{\alpha-1}\ell_s+ \ell'_s \Big)^2  \d s\bigg],\end{align*}
and
$$I_2:=    2\E^x\l[\int_0^t \e^{\frac{2}{\alpha-1}\int_0^s K(X_r)\,\d r} \ell_s \ell'_s \frac{|\nabla u_{t-s}|^2}{u_{t-s}} (X_s)\, \d s\r].$$
  So, to prove \eqref{A1}, it remains to estimate $I_2$.

By \eqref{SDE}, \eqref{N}, \eqref{P1}   and  It\^{o}'s formula, we obtain
\begin{align*}
\d \l( \e^{-2K_0s} \frac{|\nabla u_{t-s}|^2}{u_{t-s}} (X_s)\r)\overset{m}\ge  0,  \ \ s\in [0,t],
\end{align*} which together with
 $ \ff 2{\aa-1}K_0+2K_0=\ff{2\aa}{\aa-1}K_0 $ and $\ell_s'\le 0$ yields
\beq\label{PP}\beg{split}  I_2&=  \E^x \bigg[\int_0^t   (\ell_s^2)'  \,  \e^{\frac{2\aa}{\alpha-1}  K_0s } \,\E^x \Big(\e^{-2K_0s} \frac{|\nabla u_{t-s}|^2}{u_{t-s}} (X_s)\Big)\d s\bigg]\\
  &\le  \frac{|\nabla u_{t}|^2}{u_{t}} (x)\, \int_0^t   (\ell_s^2)'  \,  \e^{\frac{2\aa}{\alpha-1}  K_0 s }\d s.\end{split} \end{equation}
By $ \ell_s'\le 0,  \ell_0=1, \ell_t=0$ and integration by parts formula,  we derive
 \beq\label{PO}   0< -  \int_0^t\e^{\frac{2\aa }{\alpha-1} K_0s}  (\ell_s^2)'    \,\d s
 = 1+ \ff {2\aa} {\aa-1}  \int_0^t\ell_s^2 K_0 \e^{\ff {2\aa}{\aa-1}K_0s} \,\d s= 1+\ff{\aa}{\aa-1}\gg_{t,\aa}.\end{equation}
So,  $\gg_{t,\aa}> \ff 1 \aa-1$ and combining
with \eqref{P2-0} and \eqref{PP}  implies \eqref{A1}.
    \end{proof}

\beg{rem}\label{2.1} By using $u_0+\vv$ replacing $u_0$ then letting $\vv\to 0$, we may and do assume that $\inf u_0>0$. When $M$ is compact,  the continuous functions $\ff{|\nn u|^2}{u}$    on $[0,T]\times M$ as well as $K$ and $\si$ on $M$ are  bounded, so that  the condition \eqref{CV0'}  is easily checked. When $M$ is non-compact, if one of the following two conditions hold:
\beg{enumerate} \item[1)] $\pp M$ is either convex or empty;
\item[2)] $\II$ and $Z$ are bounded, the sectional curvature on $M$ is bounded above, and $K\in C_b(M)$,\end{enumerate}
then by \cite[Theorem 3.2.9]{Wbook}, see also \cite{HSU} for $Z=0$ and $M$ is compact, we have
\beq\label{G}|\nn u_t(x)|\le \E^x\Big[|\nn u_0|(X_t) \e^{-\int_0^t\{ K(X_s)\d s+\si(X_s)\d\L_s\}}\Big],\ \ t>0.\end{equation}
Moreover, $\E\e^{p \scr L_t}<\infty$ holds for any constants $t,p>0.$
Hence,  the boundedness of $|\nn u_0|$ implies
that of $|\nn u|$ on $[0,T]\times M$, and \eqref{CV0'} holds provided $K$ is bounded from below and there exists a constant $\dd>2$ such that
\beq\label{G1} \E^x \int_0^t|\ell_s'|^\dd\d s <\infty,\end{equation}
since $|\ell_s|=|\int_s^t\ell_s'\d s|\le \big(\int_0^t|\ell_s'|^\dd\d s\big)^{\ff 1 \dd}.$
\end{rem}

When $\pp M$ is non-convex, to apply Theorem \ref{T1} we have to estimate the exponential moment of the local time, which have been done in \cite{W05, W07, Wbook}. However, in this way one can only derive a weaker version of Li-Yau inequality where in the upper bound $u_t$ and $Lu_t$ are enlarged as $(P_t u_0^p)^{\ff 1 p} $ and
$(P_t |L u_0|^p)^{\ff 1 p}$ for some constant $p>1$.

To derive the exact Li-Yau type inequality for non-convex $M$, we present the following result by modifying the proof of Theorem \ref{T1}.  
To this end, we follow the line of  \cite{W05} to make use of a reference  function   in the following class:
$$\scr D:=\big\{1\le \phi\in C_b^2(M):\ (\II+ N\log \phi)|_{\pp M}\ge 0 \big\}.$$
Concrete choices of $\phi\in\D$ can be found in \cite{W05}  as functions of the distance to $\pp M$, which are explicitly constructed by using
bounds on the sectional curvature of $M$ and the second fundamental form of $\pp M$, see also the proof of Corollary \ref{3.3} below.

\beg{thm}\label{T1'}
Let $\pp M$ be non-convex, and let \eqref{C} hold for some   $K\in C_b(M)$. Assume that there exists   $\phi\in \D$  such that
$\|Z\phi\|_\infty<\infty$.  Then 
\beq\label{KP} \beg{split}& K_{\phi}:=2\inf_M \{K+\phi^{-1}L\phi\}>-\infty,\\
&K_{\aa,\phi} :=2\inf_M \ff{K \phi^2+\phi L\phi}{\aa-\phi^2}>-\infty,\ \ \aa>\|\phi\|_\infty^2,\end{split}
\end{equation}
and  the following assertions hold for any $t>0$ and $\ell\in C_b^1([0,t])$ with $ \ell_0=1$ and $\ell_t=0$.   
\begin{itemize}
\item [$(1)$] For any constant $\vv>0$,
  \beq\label{A1''} \begin{split}
    \ff{\phi^2|\nabla u_t|^2}{ \|\phi\|_{\infty}^2 u_t^2}  \le &\,\bigg(2 \int_0^t\ell_s |\ell_s' | \e^{(\vv-K_\phi)s} \, \d s\bigg) \frac{Lu_{t}}{u_t}\\
    & + \l(\frac{n}{2} +\frac{\|\nabla \log \phi\|_{\infty}^2}{\vv}\r)\int_0^t(\ell_s')^2\e^{(\vv-K_\phi)s} \, \d s.\end{split}
\end{equation}
\item [$(2)$]  For any  constants $\aa>\|\phi\|_\infty^2$ and $\vv>0$,   
\beq\label{A100'} \begin{split}
&   (1+{\gamma}_{t,\alpha,\phi}) \ff{\phi^2|\nabla u_t|^2}{u_t^2} -   \aa \ff {L u_t}{u_t} \\
& \le  \int_0^t \e^{(K_{\aa,\phi}-\vv) s} \big((K_{\aa,\phi}-\vv) \ell_s +2\ell_s'\big)^2 \bigg(\ff{n\aa^2 }{8}  + \ff{\aa^2 \|\nn\log\phi\|_\infty^2}{4\vv(\aa-\|\phi\|_\infty^2) } \bigg)\,\d s,\end{split}
\end{equation}
where ${\gamma}_{t,\alpha,\phi}=2\big(\frac{\alpha}{\|\phi\|_{\infty}^2} -1\big) \int_0^t |\ell_s\ell_s'| \e^{(K_{\aa,\phi}+K_\phi-\vv)s}\, \d s>0$.
\end{itemize}
\end{thm}
\beg{proof}   We may assume that $\inf u_0>0$. 
Since $\phi\in C_b^2(M)$, $\|Z\phi\|_\infty<\infty$  implies condition (3.2.15) in \cite{Wbook}, then \cite[Theorem 3.2.7]{Wbook}  for $f=u_0$ implies the boundedness 
of $|\nn u|$ on $[0,t]\times M$. So, $\ff{|\nn u|^2}{u}$ is bounded on $[0,t]\times M.$

By $\phi\in \D$ and \eqref{N} for $\sigma=-(N\log \phi)$, we obtain
\beq\label{NN} N\Big(\ff{\phi^2 |\nn u_s|^2}{u_s}\Big)\Big|_{\pp M}  = \ff{2\phi^2}{u_s}  \big(\II(\nn u_s,\nn u_s) + (N \log \phi)|\nn u_s|^2\big)\big|_{\pp M}\ge 0.\end{equation}
By \eqref{C} and the display after \eqref{N2}, we obtain
\beg{align*} &(L+\pp_s) \ff{\phi^2 |\nn u_{t-s}|^2} {u_{t-s}} = \ff{2\phi^2}{u_{t-s}}\Big(\big\|\Hess_{u_{t-s}}- \nn u_{t-s}\otimes\nn \log u_{t-s} \big\|_{HS}^2 +\Ric_Z(\nn u_{t-s},\nn u_{t-s})\Big) \\
&\qquad +\ff{(L\phi^2) |\nn u_{t-s}|^2}{u_{t-s}} +\ff{4\phi^2}{u_{t-s}}\Hess_{u_{t-s}}(\nn u_{t-s}, \nn\log\phi) -\ff{2\phi^2|\nn u_{t-s}|^2}{u_{t-s}} \<\nn \log u_{t-s},\nn \log\phi\> \\
&= \ff{2\phi^2}{u_{t-s}}\Big(\Big\|\Hess_{u_{t-s}}- \nn u_{t-s}\otimes  \nn \log \ff{u_{t-s}}\phi  \Big\|_{HS}^2    +   \Ric_Z(\nn u_{t-s},\nn u_{t-s})\Big) \\
 &\qquad  + \ff{|\nn u_{t-s}|^2(L\phi^2- 2|\nn \phi|^2)}{u_{t-s}} \\
&\ge \ff{2\phi^2}{u_{t-s}} \Big[\ff 1 m \Big(L u_{t-s} - \Big\<\nn u_{t-s},\nn   \log \ff{u_{t-s}}\phi \Big\>- Z u_{t-s} \Big)^2+  \Ric_Z(\nn u_{t-s},\nn u_{t-s})\Big] \\
&\qquad + \ff{2(\phi L\phi)|\nn u_{t-s}|^2}{u_{t-s}} \\
&\ge \ff{2\phi^2}{  u_{t-s}} \Big[\ff 1 n \Big(L u_{t-s} - \Big\<\nn u_{t-s},\nn   \log \ff{u_{t-s}}\phi \Big\>  \Big)^2+ \Ric_Z^{(n-m)}(\nn u_{t-s},\nn u_{t-s})\Big]  \\
&\qquad + \ff{2(\phi L\phi)|\nn u_{t-s}|^2}{u_{t-s}} \\
&\ge \ff{2\phi^2}{n u_{t-s}} \Big(L u_{t-s} - \Big\<\nn u_{t-s}, \nn \log \ff{u_{t-s}}\phi \Big\>  \Big)^2+    \ff{2(K\phi^2+\phi L\phi)|\nn u_{t-s}|^2}{u_{t-s}}.\end{align*}
Hence, \eqref{C}  and \eqref{KP} yield 
\beq\label{Lest}\beg{split} &(L+\pp_s) \ff{\phi^2 |\nn u_{t-s}|^2} {u_{t-s}}\\
 &\ge \ff{2\phi^2}{n u_{t-s}} \Big(L u_{t-s} - \Big\<\nn u_{t-s}, \nn \log \ff{u_{t-s}}\phi \Big\>  \Big)^2+    \ff{2(K\phi^2+\phi L\phi) |\nn u_{t-s}|^2}{u_{t-s}}\\
 &\ge  \ff{2\phi^2}{n u_{t-s}} \Big(L u_{t-s} - \Big\<\nn u_{t-s}, \nn \log \ff{u_{t-s}}\phi \Big\>  \Big)^2+ \ff{K_\phi\phi^2 |\nn u_{t-s}|^2}{u_{t-s}}.\end{split}\end{equation} 

(a) By   \eqref{NN}, \eqref{Lest} and  It\^{o}'s formula, we derive 
\beg{align*}
&\d \l( \ell^2_s \e^{(\vv-K_\phi)s} \ff{\phi^2 |\nn u_{t-s}|^2} {u_{t-s}} (X_s)\r) \notag\\
&  \overset{m}{\geq}   2\ell^2_s \e^{(\vv-K_\phi)s}   \ff{\phi^2(X_s)}{n u_{t-s}} \Big(L u_{t-s} - \Big\<\nn u_{t-s}, \nn \log \ff{u_{t-s}}\phi \Big\>  \Big)^2(X_s) \, \d s\notag\\
&\quad +2 \ell_s \ell_s'  \e^{(\vv-K_\phi)s}\phi^2(X_s)  \Big(\Big\<\nn u_{t-s}, \nn \log \ff{u_{t-s}}\phi \Big\>  -Lu_{t-s} \Big) (X_s)\, \d s\notag\\
&\quad +2 \ell_s \ell_s'  \e^{(\vv-K_\phi) s}\phi^2(X_s)  \Big( Lu_{t-s}+\langle \nabla u_{t-s}, \nabla \log \phi \rangle\Big)(X_s)\, \d s \notag\\
&\quad + \vv \ell^2_s \e^{(\vv-K_\phi)s} \ff{\phi^2 |\nn u_{t-s}|^2} {u_{t-s}} (X_s)\, \d s\notag\\
&\geq \e^{(\vv-K_\phi)s}  \phi^2 (X_s) 
 \Big((\ell_s^2)'       Lu_{t-s}   +2  \ell_s  \ell_s'  \< \nabla u_{t-s}, \nabla \log \phi\> +   \ff{\vv \ell^2_s   |\nn u_{t-s}|^2} {u_{t-s}}  -\frac{n(\ell_s')^2   u_{t-s} }2 
 \Big)(X_s)\, \d s.\end{align*}
Thus, 
\beq\label{noncovexineq}\beg{split} &\d \l( \ell^2_s \e^{(\vv-K_\phi)s} \ff{\phi^2 |\nn u_{t-s}|^2} {u_{t-s}} (X_s)\r) \\
&  \overset{m}{\geq}  \e^{(\vv-K_\phi)s}  \phi^2(X_s)  \bigg[(\ell_s^2)'     Lu_{t-s}-\Big(\frac{n}{2} +\frac{\|\nabla \log \phi\|_{\infty}^2}{\vv}\Big)(\ell_s')^2 u_{t-s} \bigg](X_s) \, \d s.
\end{split}\end{equation} 
 Combining this with
 the boundedness of $\ff{|\nn u_{t-s}|^2}{u_{t-s}}$ as explained in the beginning of the proof,  $\ell_0=1, \ell_t=0,$ and  that $u_t$ and $L u_t$ are  bounded with
 $$\E^x u_{t-s}(X_s)= u_t(x),\ \ \E^x Lu_{t-s}(X_s)= Lu_t(x),$$ 
 we derive \eqref{A1''}.  

(b) Simply denote $\bb=K_{\aa,\phi}-\vv.$ Combining \eqref{Lest} with \eqref{P1}, \eqref{NN}, $N Lu_{t-s}|_{\pp M}=0$ and $(L+\pp_s) Lu_{t-s}=0$ for $s\in [0,1)$, and applying It\^o's formula to \eqref{SDE}, we derive 
\beg{align*} & \d\bigg\{\e^{\bb s} \ell_s^2 \Big(\ff{\phi^2 |\nn u_{t-s}|^2} {u_{t-s}}-\aa L u_{t-s}\Big)(X_s)\bigg\} \\
&\overset{m}{=} \e^{\bb s}\bigg\{\ell_s^2(L+\pp_s) \ff{\phi^2 |\nn u_{t-s}|^2} {u_{t-s}} + \big(\bb \ell_s^2 + 2\ell_s\ell_s'\big)  \Big(\ff{\phi^2 |\nn u_{t-s}|^2} {u_{t-s}}-\aa L u_{t-s}\Big) \bigg\}(X_s)\,\d s\\
&\ge  \e^{\bb s}\bigg\{\ff{2\ell_s^2\phi^2}{n u_{t-s}}  \Big(L u_{t-s} - \Big\<\nn u_{t-s}, \nn \log \ff{u_{t-s}}\phi \Big\>  \Big)^2 \\
&\quad  -\aa (\bb \ell_s^2+2\ell_s\ell_s')  \Big(L u_{t-s} 
- \Big\<\nn u_{t-s}, \nn \log \ff{u_{t-s}}\phi \Big\>  \Big)  
   -\aa (\bb \ell_s^2+2\ell_s\ell_s') \Big\<\nn u_{t-s},  \nn \log \ff{u_{t-s}}\phi \Big\>  \\
 &\quad  \qquad \qquad \qquad \qquad  \qquad \qquad \qquad +   \Big(\big[\bb+2K+2\phi^{-1}L\phi\big]\ell_s^2+2\ell_s\ell_s' \Big) \ff{\phi^2|\nn u_{t-s}|^2}{u_{t-s}} \bigg\}(X_s)\,\d s\\
 &\ge  \e^{\bb s}\bigg\{- \ff{n\aa^2 (\bb \ell_s +2 \ell_s')^2}{8\phi^2} u_{t-s} 
   +\big(2K\phi^2+\phi L\phi+\bb(\phi^2-\aa )\big)\ell_s^2 \ff{  |\nn u_{t-s}|^2}{u_{t-s}}\\
 &\qquad \quad  -2 \ell_s\ell_s'(\aa -\phi^2)    \ff{|\nn u_{t-s}|^2}{u_{t-s}} -\aa |\bb \ell_s^2 +2\ell_s \ell_s'|\|\nn\log\phi\|_\infty |\nn u_{t-s}| \bigg\} (X_s)\,\d s.  \end{align*} 
Since $\ell_s\ell_s'<0$, $\aa>\phi\ge 1$, and  $\bb= K_{\aa,\phi}-\vv$  yields 
$$   2 K \phi^2 +2\phi  L\phi  +  \bb (\phi^2-\aa )  \ge  (\bb-K_{\aa,\phi}) (\phi^2-\aa )  \ge  (\aa-\|\phi\|_\infty^2)\vv,$$
this further implies 
\beg{align*}  & \d\bigg\{\e^{\bb s} \ell_s^2 \Big(\ff{\phi^2 |\nn u_{t-s}|^2} {u_{t-s}}-\aa L u_{t-s}\Big)(X_s)\bigg\} \notag\\ &\ge   
-  \e^{\bb s}  (\bb\ell_s+2\ell_s')^2 \bigg(\ff{n\aa^2 }{8}  + \ff{\aa^2 \|\nn\log\phi\|_\infty^2}{4(\aa-\|\phi\|_\infty^2)\vv}\bigg)u_{t-s}(X_s)\,\d s\notag \\
& \quad +2|\ell_s\ell_s'|(\alpha\phi^{-2}-1) \e^{ \beta s}     \ff{\phi^2|\nn u_{t-s}|^2}{u_{t-s}}(X_s)\, \d s.
\end{align*} 
Combining this with $\E u_{t-s}(X_s)= u_t(x) $ for $X_0=x$,  
 the boundedness of $\ff{|\nn u_{t-s}|^2}{u_{t-s}}, u_{t-s}, L u_{t-s}$ as explained in the beginning of the proof,  and $\ell_0=1, \ell_t=0,$  
 we derive
\beq \label{Es}\beg{split}  &\ff{\phi^2 |\nn u_t|^2}{u_t} (x) -\aa L u_t(x) \\
&\le u_t(x)\bigg(\ff{n\aa^2 }{8}  + \ff{\aa^2 \|\nn\log\phi\|_\infty^2}{4(\aa-\|\phi\|_\infty^2)\vv}\bigg) \int_0^t \e^{\bb s}  (\bb\ell_s+2\ell_s')^2   \d s\\
&\quad - 2 \int_0^t |\ell_s\ell_s' | \e^{\beta  s}   \E\Big[ (\alpha-\phi^{2})  \ff{|\nn u_{t-s}|^2}{u_{t-s}}\Big](X_s)\, \d s.\end{split}\end{equation} 
By  \eqref{noncovexineq} for $\ell=1$ and $\vv\downarrow 0$,  we have 
\begin{align*}
\d \l(\e^{-K_{\phi} s}  \frac{\phi^2|\nabla u_{t-s}|^2}{u_{t-s}}(X_s)\r)\geq 0,
\end{align*}
so that 
\begin{align*}
& -2 \int_0^t |\ell_s\ell_s' | \e^{\beta  s}   \E\Big[ (\alpha-\phi^{2})  \ff{|\nn u_{t-s}|^2}{u_{t-s}}\Big](X_s)\, \d s\\
&\le  - 2\int_0^t|\ell_s\ell_s'|(\alpha \|\phi\|_{\infty}^{-2}-1) \e^{(\beta+K_{ \phi}) s} \E \l[ \e^{-K_{\phi} s} \phi^2\ff{|\nn u_{t-s}|^2}{u_{t-s}}(X_s)\r]\, \d s\\
&\le -2(\alpha  \|\phi\|_{\infty}^{-2}-1)\bigg(\int_0^t|\ell_s\ell_s' |\e^{(\beta+K_{\phi}) s}\, \d s\bigg) \ff{|\phi^2\nn u_{t}|^2}{u_{t}}(x).  
\end{align*}
 Combining this with \eqref{Es},   $\bb=K_{\aa,\phi}-\vv$ and the definition of $\gg_{t,\aa,\phi},$
 we finish the proof of \eqref{A100'}. 
\end{proof} 

\section{Global  Li-Yau type estimates   }

We will present explicit   Li-Yau type estimates by using Theorem \ref{T1} and Theorem \ref{T1'} for the convex and non-convex cases respectively.

\subsection{$\pp M$ is convex or empty}

By  Remark \ref{2.1},    \eqref{CV0'} holds for all $\ell\in C_b^1([0,t])$ provided $K$ is bounded from below and $\pp M$ is  either convex or empty.
Therefore, estimates \eqref{J0} holds for $\si=0$, and \eqref{A1} holds for $\ell_s'\le 0$.   By taking specific choices of $\ell_s$ in
these estimates,  we    present explicit Li-Yau type inequalities  in  the following Corollaries \ref{C1} and \ref{cor11}, where
\beg{enumerate} \item[$\bullet$]    \eqref{J2}  improves  \eqref{J'} when $t>\ff{\pi}K$, and   is sharp for small time as shown by \cite[Corollary 2.3]{BBG}; 
  \item[$\bullet$]     \eqref{A2} and \eqref{A2'} are new even for $Z=0$ and $\pp M=\emptyset$;
  \item[$\bullet$]     \eqref{BBG} is due to \cite{BBG} for $\pp M=\emptyset$, which improves   a number of  classical bounds recalled in the introduction as shown in \cite[Section 5]{BBG}.    
  \end{enumerate} 
 When $\pp M$ is strictly convex such that $\si$ is a positive constant,
  $\E^x[\e^{-2\si \scr L_s}]$ decays exponentially fast as $s\to\infty$ according to   \cite[Lemma 3.1]{W07},  so that \eqref{J0} may provide better estimates than those presented in Corollary \ref{C1}.

 \beg{cor}  \label{C1} Assume that $\pp M$ is either empty or convex, and let $\eqref{C}$ hold for a constant $K\in\R.$ Then   the inequality \eqref{J0}  holds for $\si=0$, which implies the following estimates.
 \beg{enumerate}
 \item[$(1)$] For any $t>0$,
 \beq\label{J2} \ff{L u_t}{u_t}\beg{cases} \le \ff n {4t}\big[(Kt)\land\pi + \ff{\pi^2}{(Kt)\land \pi}\big],\ &\text{if}\ K>0, \\
  \ge  -\ff{n}{4t} \big[\pi\lor (-Kt) + \ff{\pi^2}{\pi\lor (-Kt)}\big], \  &\text{if}\ K\le 0.\end{cases} \end{equation}
  \item[$(2)$] When $K\ne 0$, for  any constant $\aa  \in \R$ such that
  $  \ff{1+\aa}{Kt}\ge \ff{9\pi^2-64}{9\pi^2},$     let
$$\bb_{t,\aa}:= \bigg( \ff{1+\aa}{Kt}- \ff{9\pi^2-64}{9\pi^2}  \bigg)^{\ff 1 2}- \ff 8{3\pi}.  $$ Then
\beq\label{A2}   \frac{|\nabla u_t|^2}{u_t^2}- \ff {\aa L u_t}{u_t}  \le
\ff n 2\bigg( \ff {K(\aa-1)} {2}   +\ff{(1+\aa)\pi^2}{2Kt^2} -\ff{2\pi \bb_{t,\aa}}t -\ff{3\pi^2}{8t}\bigg). \end{equation}
  \item[$(3)$] When $K>0$, for any   constant $K'\ge K$ such that $\Ric_Z\ge K'$,  we have 
\beq\label{A2'}  \frac{|\nabla u_t|^2}{u_t^2}\le  \ff n 2 \bigg(\ff{\pi^2K}{2[1\land (Kt)]^2} -\ff K2- \ff{3K\pi^2}{8[1\land (Kt)]}\bigg)\e^{-2K'(t-K^{-1})^+},\ \ t>0. \end{equation}
\item[$(4)$] $\eqref{BBG}$ holds.
 \end{enumerate} \end{cor}

\beg{proof} When $\pp M$ is empty or convex,  we may take $\si=0$ such that for any deterministic $\ell_s$ with
    $\ell_0=1$ and $ \ell_t=1$,
$$\int_0^t (\ell_s^2)' \e^{-2\int_0^s[ K (X_r)\,\d r+\si(X_r)\,\d\scr L_r]} \, \d s= - 1+ 2 \int_0^t K(X_s) \e^{-2\int_0^s K (X_r)\,\d r} \ell_s^2   \, \d s.$$
When $K$ is a constant, then   \eqref{J0} reduces to
\beq\label{J00'}
\frac{|\nabla u_t|^2}{u_t^2}
 \leq \frac{n}{2}  \left[\int_0^t |\ell_s'|^2 \e^{-2Ks}\, \d s\right]   +
\left(1-2 K  \int_0^t\ell_s^2 \e^{-2Ks} \d s\right)\ff{L u_t}{u_t}.\end{equation}

(1) For a fixed constant $a\in\R $, let
$$\ell_s:= \e^{Ks} \bigg( \cos\ff{\pi s}{2t} + a \sin\ff{\pi s}t\bigg),\ \ \ s\in [0,t].$$ We have
\beq\label{AB1}    \int_0^t \ell_s^2 \e^{-2Ks}\,\d s =  \int_0^{t} \Big[\cos^2 \ff{\pi s}{2t} + a^2 \sin^2  \ff{\pi s}t + 2 a \Big(\cos \ff{\pi s}{2t} \Big)   \sin\ff{\pi s}t \Big]\,\d s,\end{equation}
\beq\label{AB2} \beg{split} & \int_0^t |\ell_s'|^2 \e^{-2Ks}\,\d s = K^2\int_0^t \ell_s^2 \e^{-2Ks}\d s-K \\
 &\qquad  +
\int_0^t \Big[\ff{\pi^2 a^2 }{t^2} \cos^2 \ff{\pi s}t+ \ff{\pi^2}{4t^2} \sin^2\ff{\pi s}{2t}- \ff{\pi^2 a}{t^2} \Big(\cos\ff{\pi s}t\Big)\sin\ff{\pi s}{2t} \Big]\,\d s.\end{split}\end{equation}
Noting that
\beg{align*} &\int_0^t  \sin^2  \ff{\pi s}t\, \d s=\ff t{\pi} \int_0^\pi \sin^2\theta\, \d\theta= \ff{t}2=\int_0^t  \cos^2  \ff{\pi s}t\, \d s,\\
&\int_0^t\cos^2 \ff{\pi s}{2t}\,\d s= \ff {2t}\pi\int_0^{\ff\pi 2} \cos^2\theta\,\d\theta =\ff t 2= \int_0^t\sin^2 \ff{\pi s}{2t}\,\d s,\\
& \int_0^t \Big(\cos \ff{\pi s}{2t} \Big)   \sin\ff{\pi s}t\, \d s =\ff {2t}\pi \int_0^{\ff\pi 2} (\cos\theta)\sin(2\theta)\,\d\theta =- \ff {4t}\pi \int_0^{\ff\pi 2} (\cos^2\theta)\,\d \cos \theta =\ff{4t}{3\pi},\\
&\int_0^t \Big(\cos \ff{\pi s}t\Big) \sin\ff{\pi s}{2t}\,\d s=\ff{2t}\pi\int_0^{\ff\pi 2} (\cos (2\theta)) \sin\theta \,\d \theta= - \ff{2t}\pi\int_0^{\ff\pi 2} \big(2\cos^2  \theta -1\big)\,\d\cos\theta
= -\ff{2t}{3\pi}.
\end{align*} Combining these with \eqref{AB1} and \eqref{AB2}, we obtain
\beq\label{A*} \beg{split} &2K  \int_0^t \ell_s^2 \e^{-2Ks}\d s= Kt (1+ a^2)+ \ff{ 16 Kta}{3\pi},\\
&   \int_0^t |\ell_s'|^2 \e^{-2Ks}\,\d s=  \ff K {2}\Big(Kt (1+ a^2)+ \ff{16 Kta}{3\pi}\Big)   -K+  \ff{\pi^2a^2}{2t}+\ff{\pi^2}{8t}+\ff{2\pi a}{3t}.\end{split}\end{equation}
  Substituting into \eqref{J00'} and letting $a\to\infty$, we derive that when $K>0$,
  $$\ff{Lu_t}{u_2}\le \ff n 2 \lim_{a\to\infty} \ff{ \int_0^t |\ell_s'|^2 \e^{-2Ks}\,\d s}{2K  \int_0^t \ell_s^2 \e^{-2Ks}\,\d s}= \ff n 4  \Big(  K   + \ff{\pi^2}{Kt^2}\Big),$$
  while for $K<0$,
    $$\ff{Lu_t}{u_2}\ge \ff n 2 \lim_{a\to\infty} \ff{- \int_0^t |\ell_s'|^2 \e^{-2Ks}\,\d s}{-2K  \int_0^t \ell_s^2 \e^{-2Ks}\,\d s}= \ff n 4  \Big(  K   + \ff{\pi^2}{Kt^2}\Big).$$

 Since for any $k\in (0,K]$ the condition \eqref{C} holds for $k$ in place of $K$,  so that the above estimates   also hold  for  $k\le K$ replacing $K$. By taking $k=K\land \ff\pi t$   for $K>0$, and
 $K=K\land (-\ff \pi t)$ for $K<0$,  the above estimates for $k$ replacing $K$ implies  \eqref{J2}.

  (2)  By the definition of $\bb_{t,\aa}$ and \eqref{A*}, we obtain
  $$\bb_{t,\aa}^2= \ff {1+\aa} {Kt}-1-\ff{16\bb_{t,\aa}}{3\pi},$$
  and
  $$2K  \int_0^t \ell_s^2 \e^{-2Ks}\, \d s= Kt (1+ \bb_{t,\aa}^2)+ \ff{ 16 Kt\bb_{t,\aa}}{3\pi}=1+\aa.$$
 Combining these with \eqref{A*} and \eqref{J00'} for $a=\bb_{t,\aa}$, we derive
  \beg{align*} &\ff{|\nn u_t|^2}{u_t^2} -\ff {\aa L u_t}{u_t} \le \ff n 2 \int_0^t  |\ell_s'|^2 \e^{-2Ks}\,\d s \\
  &=\ff n 2\bigg(\ff {(1+\aa)K} {2} -K +\ff{\pi^2}{2t}  \Big(\ff 1 {Kt}-1-\ff{16 \bb_{t,\aa}}{3\pi}\Big) +\ff{\pi^2}{8t}+\ff{2\pi \bb_{t,\aa}}{3t}\bigg)\\
  & \le \ff n 2\bigg( \ff {(\aa-1)K} {2 }   +\ff{\pi^2}{2Kt^2} -\ff{2\pi  \bb_{t,\aa}}t -\ff{3\pi^2}{8t}\bigg).\end{align*}
  Then \eqref{A2} holds.

 (3)  When $t\le \ff 1 K$ and $\aa=0$, we have $ \bb_{t,\aa}\ge 0$,    so that \eqref{A2} implies
  \beq\label{TA}\beg{split}  \frac{|\nabla u_t|^2}{u_t^2}   & \le
\ff n 2\bigg( \ff {(\aa-1)K} {2 }    +\ff{ \pi^2}{2Kt^2} -\ff{2\pi \bb_{t,\aa} } t -\ff{3\pi^2}{8t}\bigg)\\
&\le \ff n 2\bigg( \ff{ \pi^2}{2Kt^2}  -\ff K 2  -\ff{3\pi^2}{8t}\bigg).\end{split} \end{equation}   Let
  $P_t$ be the (Neumann) semigroup generated by $L$. Then   $u_t= P_{(t-K^{-1})^+} u_{t\land K^{-1}}$ and  it is well known that $\Ric_Z\ge K'$ implies
  $$|\nn u_t|=|\nn P_{(t-K^{-1})^+} u_{t\land K^{-1}} | \le \e^{-K' (t-K^{-1})^+} P_{(t-K^{-1})^+} |\nn u_{t\land K^{-1}}|.$$
  Combining this with \eqref{TA} for $t\land K^{-1}$ replacing $t$, we derive  \eqref{A2'}.

(4) By \eqref{J'}, $\lambda:=1-\frac{4}{nK}\frac{ L  u_t}{u_t} >  -\frac{\pi^2}{K^2t^2}.$
Choose $\ell_s= h_s\e^{Ks} $ for  $$ h_s:=\begin{cases}
\frac{\sinh(K\sqrt{\lambda} (t-s))}{\sinh(K\sqrt{\lambda} t)},& \ \ \  \lambda >0;\\
\frac{t-s}{t},&\ \ \  \lambda=0; \\
\frac{\sin(K\sqrt{-\lambda }(t-s))}{\sin(K\sqrt{-\lambda }t)}, &\ \ \ \lambda\in (-\frac{\pi^2}{K^2t^2},0).
\end{cases}$$
We have $ \ell'_s= h'_s\e^{Ks}+Kh_s\e^{Ks}$. So, \eqref{J00'} implies
\begin{align*}
&\ff{|\nn u_t|^2}{u_t^2} \le \frac{n }{2}  \int_0^t |\ell_s'|^2 \e^{-2Ks}\, \d s   -2  \l[ \int_0^t\ell_s \ell'_s\e^{-2Ks} \, \d s\r] \frac{ L  u_t}{u_t} \notag\\
&=\frac{n }{2}  \int_0^t ( h'_s+Kh_s)^2\, \d s   -\frac{2 L  u_t}{u_t}    \int_0^th_s( h'_s+Kh_s) \, \d s \notag\\
&= \frac{n }{2} \int_0^t (|h'_s|^2+2Kh_s h'_s+K^2 h_s^2)\, \d s   -\frac{2 L  u_t}{u_t}  \int_0^t(h_s h'_s+Kh_s^2) \, \d s  \notag\\
&= \frac{n }{2}  \int_0^t |h'_s|^2\, \d s +\l( Kn-\frac{2 L  u_t}{u_t} \r)   \int_0^t h_s h'_s\,\d s
  + \l(\frac{K^2n}{2}-2K \frac{ L  u_t}{u_t}\r)  \int_0^t h_s^2\, \d s.
\end{align*}
It is easy to see that
\begin{align*}
 \int_0^t h_s h'_s\,\d s = \int_0^t h_s \,d h_s = \frac{h_t^2-h_0^2}{2} =-\frac{1}{2}.
\end{align*}
Moreover, by $h_s''= K^2\ll h_s$ for $\ll:= 1-\ff 4{nK}\ff{L u_t}{u_t}$ due to the definition of $h_s$, one has
 \begin{align*}
  \int_0^t |h'_s|^2\, \d s =   \int_0^t  h'_s\, d h_s
 =  h'_sh_s|^t_0- \int_0^t h_s'' h_s\, \d s
 =-   h'_0- K^2\l(1-\ff 4{nK}\ff{L u_t}{u_t}\r) \int_0^t   h_s^2\, \d s.
\end{align*} We then conclude that
 \begin{align*}   \ff{|\nn u_t|^2}{u_t^2}
 & \le \frac{ L  u_t}{u_t}
 -\frac{n}{2}  h_0'-\frac{Kn}{2}
 -\frac{nK^2}{2} \l(1-\ff 4{nK}\ff{L u_t}{u_t}\r)\int_0^t   h_s^2\, \d s + \l(\frac{K^2n}{2}-2K \frac{ L  u_t}{u_t} \r)  \int_0^t h_s^2\,\d s\\
 &= \frac{ L  u_t}{u_t}  -\frac{n}{2}  h_0'-\frac{Kn}{2}.
\end{align*}
This implies \eqref{BBG} by noting that $h_0'=-\Phi_t(\ll).$
\end{proof}

Estimates in the next corollary   are implied by $\eqref{A1}$, where \eqref{A2B} is new, and \eqref{NE} improves \eqref{LY1} by
noting that for $K< 0$ and $\aa>1$,
$$\ff K 4 \coth \Big(\ff{Kt}{2(\aa-1)}\Big)< \ff{K^-}2 +\ff{\aa-1}{2t},$$
and   as in \eqref{PO} we have
$$   \ff{2K\int_0^t (1- \e^{-\ff{Ks}{\aa-1}} )^2\d s}{\aa(1-\e^{-\ff{Kt}{\aa-1}})^2}>\ff 1 \aa-1.$$

\begin{cor}\label{cor11}
Assume that $\pp M$ is either empty or convex, and let $\eqref{C}$ hold for a constant $K\in\R.$  Then the inequality
$\eqref{A1}$ holds for $\sigma=0$ and  implies the following estimates.
\beg{enumerate}
 \item[$(1)$] For any constant $\aa>1$
\beq\label{NE} \beg{split}&\l(1+ \ff{2K\int_0^t (1- \e^{-\ff{Ks}{\aa-1}} )^2\d s}{\aa(1-\e^{-\ff{Kt}{\aa-1}})^2}\r)
\frac{|\nabla u_t|^2}{u_t^2}\\
&   \le \ff{  Lu_t}{u_t}+ \ff{nK\aa } {4(\aa-1)}\coth\Big(\ff{Kt}{2(\aa-1)}\Big),\ \ t>0.\end{split}\end{equation}
\item[$(2)$] For any $t>0$ and $\aa\ge 1+ K^-t$,
 \beq\label{A2B}    \Big(1+\ff 2 {3\aa} K t\Big) \frac{|\nabla u_t|^2}{u_t^2}   \le \ff{   L u_t}{u_t} + \ff {n \aa  } {2t}.  \end{equation}
 Consequently, when $K>0$,
 \beq\label{A2BB}    \Big(1+\ff 2  3  K t\Big) \frac{|\nabla u_t|^2}{u_t^2}   \le \ff{L u_t}{u_t} + \ff {n } {2t},\ \ t>0.  \end{equation}
\end{enumerate}

\end{cor}

\begin{proof}
If $\eqref{C}$ holds for a constant $K\in\R$ and $\pp M$ is convex or empty,    then the inequality \eqref{A1}  reduces to
 \beq\label{A1'}
  (1+\gamma_{t,\alpha}) \ff{|\nabla u_t|^2}{u_t^2} (x) -   \frac{\aa L u_t (x)}{u_t(x)}
  \le  \frac{n\alpha^2 }{2 } \int_0^{t}\Big(\frac{K}{\alpha-1}\ell_s+\ell'_s\Big)^2\e^{\frac{2}{\alpha-1}Ks  }\,\d s,
\end{equation}
where $$\gg_{t,\aa}   :=  2 K \int_0^{t}  \ell_s^2 \e^{\frac{2\aa}{\alpha-1}K s }  \,\d s>\ff 1\aa -1.$$
(1)
 Let $\ell_s=\frac{\int_0^{t-s} \e^{ \frac{K r}{\alpha-1} }\,\d r}{\int_0^t \e^{ \frac{K r}{\alpha-1} }\,\d r}, s\in [0,t].$
Then $ \ell_0=1,\ell_t=0$ and
\begin{align*}
\ell_s'=\frac{- \e^{\frac{K(t-s)}{\alpha-1}}}{\int_0^t \e^{ \frac{K r}{\alpha-1} }\,\d r}=\frac{-K}{\alpha -1}\ell_s-\frac{1}{\int_0^t \e^{ \frac{K r}{\alpha-1} }\,\d r}\le 0.
\end{align*}
So,
\begin{align*}
\int_0^{ t}\e^{\frac{2Ks}{\alpha-1}}\l(\ell_s'+\frac{K}{\alpha-1}\ell_s\r)^2\, \d s
&=\frac{\int_0^t \e^{\frac{2Ks}{\alpha-1}}\, \d s}{(\int_0^t \e^{ \frac{K r}{\alpha-1} }\, \d r)^2}=\frac{K}{2(\alpha-1)}\coth\l(\frac{Kt}{2(\alpha-1)}\r).
\end{align*}
Moreover, according to the definition of $\gamma_{t,\alpha}$,
$$
 \gg_{t,\aa}= \ff {2K} \aa \int_0^t \ell_s^2 \e^{-2Ks}\d s=\ff{2K\int_0^t (1- \e^{-\ff{Ks}{\aa-1}} )^2\,\d s}{\aa(1-\e^{-\ff{Kt}{\aa-1}})^2}.
 $$
Then \eqref{NE} follows by combining these estimates   with  \eqref{A1'}.

(2) For $\aa\in (1,\infty)$, let
$$\ell_s= \e^{-\ff{Ks}{\aa-1}}\ff{t-s}t,\ \ s\in [0,t].$$
Then $\ell_0=1,\ell_t=0,$ and when $\aa\ge 1+K^-t$ we have
\beg{align*} &\ell_s'= -\ff{K}{\aa-1}\ell_s-\ff 1 t \e^{-\ff{Ks}{\aa-1}}\le 0,\\
&\ell_s'+\ff{K}{\aa-1}\ell_s=-\ff 1 t \e^{-\ff{Ks}{\aa-1}}.\end{align*}
Then
\beg{align*} & \int_0^t \e^{\ff{2K}{\aa-1}s} \Big(\ff{K}{\aa-1}\ell_s+\ell_s'\Big)^2\d s = \int_0^t \ff 1 {t^2} \d s=\ff 1 t,\\
&\gg_{t,\aa}=\ff{2K}\aa \int_0^t \Big(\ff{t-s}t\Big)^2\d s = \ff{2Kt}{3\aa}.  \end{align*}
Thus,   \eqref{A1'} implies \eqref{A2B}, and further \eqref{A2BB} by letting $\aa\downarrow1$ when $K>0$.

\end{proof}

 \subsection{$\pp M$ is non-convex  }

Let $\rr_\pp$ be  the Riemannian distance to $\pp M$.  By choosing specific   $\phi\in\D$ as function of $\rr_\pp$, we obtain explicit Li-Yau type inequality from Theorem \ref{T1'} with geometry quantities, by choosing
test functions $\ell$ as in the proofs of Corollary \ref{C1} and Corollary \ref{C1}.  

\beg{cor}\label{3.3}  Assume that $\eqref{C}$ and $\eqref{S}$ hold  for some constants $K\in \R$ and $\si<0$, and there exist constants $k,\theta,\si \ge 0$ and $r_0>0$ such that $\rr_\pp$ is smooth on
$\pp_{r_0}M:=\{x\in M: \rr_\pp(x)\le r_0\}$,   
$   \II\le \theta$,    the sectional curvature of $M$ on $\pp_{r_0}M$ is bounded above by $k$,   and $| Z\rr_\pp|$ is bounded on $\pp_{r_0}M$.
Let 
\beg{align*}&h_s:= \cos\ss k\,s -\ff\theta{\ss k}\sin\ss k\,s,\ \ s\ge 0,\\
&\dd:= -\ff{\si (1-h_{r_0})^{d-1}}{\int_0^{r_0}(h_s-h_{r_0})^{d-1}\d s},\\
&\kk:= 1+\dd \int_0^{r_0} (h_s-h_{r_0})^{1-d}\int_s^{r_0} (h_r-h_{r_0})^{d-1}\,\d r,\\
&\gg:=  \dd (1- h_{r_0})^{1-d} \int_0^{r_0} (h_s-h_{r_0})^{d-1}\,\d s, \end{align*}
where for $k=0$ the function $h_s$ is defined by the limit as $k\downarrow 0$. 
Then $\eqref{A1''}$ and $\eqref{A100'}$ hold for $ \|\phi\|_\infty=\kk, \|\nn \log\phi\|_\infty=\gg,$  and
\beg{align*}  & K_\phi=-2(K-\delta +\si\|Z \rho_{\partial}\|_{\partial_{r_0}M} ),\\
& K_{\aa,\phi}= - \ff{2\kk^2(\delta -\si \|Z \rho_{\partial}\|_{\partial_{r_0}M}    + K^-)}{\aa-\kk^2},\ \ \aa>\kk^2.\end{align*}
\end{cor}

\beg{proof} According to the proof of Theorem 3.2.9(2) in \cite{Wbook}, we may choose $\phi= \varphi\circ\rr_\pp$, where
$$\varphi(r):= 1+\dd\int_0^r  (h_s-h_{r_0})^{1-d}\int_{s\land r_0}^{r_0} (h_a-h_{r_0})^{d-1}\,\d a \, \d s,\ \ r\ge 0.$$
Then $\frac{1}{\phi}\DD \phi\ge -\dd$, so that 
\begin{align*}
&2\inf_M \{K+\phi^{-1}L\phi\}\geq -2(K-\delta+\sigma \|Z\rho_{\partial}\|_{\partial_{r_0}M});\\
&\inf_M \ff{2(\phi L\phi+K\phi^2)}{\aa-\phi^2} \ge -\ff{2\kk^2(\delta -\sigma\|Z \rho_{\partial}\|_{\partial_{r_0}M}  + K^-)}{\aa-\kk^2}.
\end{align*}
\end{proof}

\section{Local Li-Yau Estimates   }

When $M$ is non-compact and does not satisfy any conditions in Remark \ref{2.1},
to estimate $|\nn u_t(x)|$ we may first  restrict  our calculus  to a compact domain $D$ containing $x$ as an interior point,
by using stochastic analysis as in the proof of Theorem \ref{T1} before the exit time of $X_t$ from $D$.
Under this restriction, we may estimate $|\nn u_{t-s}(X_s)|$  by using bounded   geometry quantities on $D$, so that the condition  \eqref{CV0'} is replaced by a suitable choice of the test function $f$ satisfying \eqref{G0} below.

\beg{thm}\label{T2} Let   $x\in M$, and let $D$ be a compact domain in $M$ such that $x\in D^o:=D\setminus\pp D$ and when $D\cap \pp M\ne\emptyset$
\beq\label{S'} \si|_{D\cap \pp M}\ge 0.\end{equation}
  Let $K_D\ge 0$ be a constant such that $\eqref{C}$ holds on $D$ for $K=-K_D$.
Let $f\in C_b^2(D)$ such that
\beq\label{G0} f|_D\le 1,\ \  f(x)=1, \ \ f|_{\pp D}=0, \ \ f|_{D^0}>0,\ \ Nf|_{D\cap \pp M}\ge 0,\end{equation}
where the condition $Nf|_{D\cap \pp M}\ge 0$ applies only when $\pp M$ exists.
 Then  the following estimates hold.
 \beg{enumerate}\item[$(1)$] For any constant $\vv>0$, let
 $$\bb_{\vv,f}:= \sup_D \bigg\{2K_D-2f Lf +\Big(6+\ff{(1+\vv)^2  n}\vv  \Big)|\nn f|^2\bigg\}.$$ Then for any $t,\vv>0$,
\beq\label{G3} \beg{split}  \ff{|\nn u_t|^2}{u_t^2}(x)\le &\,\ff{n(1+\vv)^2\bb_{\vv,f}}{2(1-\e^{-\bb_{\vv,f}t})} \\
&  +  \ff{2(1+\vv)\bb_{\vv,f}\int_0^t (\e^{-2\bb_{\vv,f}s}-\e^{-\bb_{\vv,f}(s+t)})\e^{2K_Ds}\d s}{(1-\e^{-\bb_{\vv,f}t})^2}  \frac{L u_t}{u_t}(x).\end{split} \end{equation}
\item[$(2)$] For any constant $\aa>1$, let
$$ \tt\bb_{\aa,f}:= \sup_D\bigg\{\Big(6+\ff{n\aa^2}{\aa-1}\Big)|\nn f|^2 -2f Lf -\ff{2K_D}{\aa-1} f^2\bigg\}.$$   Then
for any $t>0$,
\beq\label{D4} \beg{split}& \ff{|\nabla u_t|^2}{u_t^2}(x)\le \aa \frac{Lu_t}{u_t}(x)
 +\ff{n\aa^2}2 \l(\ff{K_D}{\aa-1}  +\frac{\tt\bb_{\aa,f}}{1-\e^{-\tt\bb_{\aa,f} t}}\r).\end{split}\end{equation}
 \end{enumerate}
\end{thm}

\beg{proof} (a) We follow the line of \cite{TW98} by making a suitable time change $X_{\tau(t)}$ of the (reflecting) diffusion process $X_t$, where $\{\tau(t)\}_{t\ge 0} $ are stoping times satisfying  
$$\tau(t)\le \tau_D:=\inf\big\{t\ge 0: X_t\in\pp D\big\}<\infty\ \text{a.s.}.$$
To this end,    let
\beg{align*}
&T(t):=\int_{0}^{t}f^{-2}(X_s )\, \d s,\ \  t\in [0,\tau_D],\\
 &\tau(t):=\inf\{s\geq 0: T(s)\geq t\},\ \ t\ge 0.\end{align*}
We have $T(\tau(t))=t$ for all $t\ge 0$,     $\tau(T(t))=t$ for $t\in [0,\tau_D]$ and
\beq\label{TO} \d T(t)=f^{-2}(X_t)\,\d t,\ \ \d\tau(t)= f^2(X_{\tau(t)})\,\d t. \end{equation}
The time-changed diffusion $X_t':=X_{\tau(t)}$ is generated by $L':=f^2L$ which never hits the boundary $\pp D$,  see \cite{TW98}.

Since $f\le 1$, we have $T(t)\ge t$ and $\tau(t)\le t$.
For fixed $t>0$,   let
\beq\label{HH} h_s:= \ff{\e^{-\bb_{\vv,f}s}-\e^{-\bb_{\vv,f}t}}{1- \e^{-\bb_{\vv,f}t}},\ \ s\in [0,t].\end{equation}
Then $h\in C_b^1([0,t])$ satisfying
\beq\label{H} h_0=1,\ \ h_t=0,\ \ h_s''=-\bb_{\vv,f} h_s',\ \ s\in [0,t].\end{equation}
By  \eqref{N} with $\si\ge 0$ due to \eqref{S'}, \eqref{P1} for $K=-K_D$, and  It\^{o}'s formula since $X_{\cdot}$ is a solution to \eqref{SDE},   we obtain
\begin{align*}
&\d\left(h_s^2 \e^{ 2K_Ds }  \frac{|\nabla u_{t-\tau(s)}|^2}{u_{t-\tau(s)}} (X_{\tau(s)})\right)   \\
& \overset{m}{\ge}  2K_D(1-f^2(X_{\tau(s)}))  h_s^2\e^{  2K_Ds } \frac{|\nabla u_{t-\tau(s)}|^2}{u_{t-\tau(s)}} (X_{\tau(s)})\, \d s  \\
&\qquad + 2h_sh_s' \e^{2K_Ds} \frac{|\nabla u_{t-\tau(s)}|^2}{u_{t-\tau(s)}} (X_{\tau(s)}) \,\d s\\
&\qquad + \ff{2 h_s^2f^2}{nu_{t-\tau(s)}}(X_{\tau(s)} )\e^{ 2K_Ds}  \Big(L u_{t -\tau(s)}-\ff{|\nn u_{t-\tau(s)}|^2}{u_{t-\tau(s)}} \Big)^2 (X_{\tau(s)})\,\d s.\end{align*}
Noting that $f\le 1$ and for any $\vv\in (0,1)$,
\beg{align*} & 2h_sh_s'  \ff{|\nn u|^2}u + \ff{2h_s^2f^2}{nu}    \Big(L u  -\ff{|\nn u |^2}{u} \Big)^2\\
&= -2 \vv h_sh_s'  \ff{|\nn u|^2}u + 2(1+\vv) h_sh_s'  Lu \\
&\quad + 2(1+\vv) h_sh_s'  \Big( \ff{|\nn u |^2}{u} -Lu \Big) +   \ff{2h_s^2f^2}{nu}   \Big(L u  -\ff{|\nn u |^2}{u} \Big)^2\\
&\ge -2 \vv h_sh_s'  \ff{|\nn u|^2}u + 2(1+\vv) h_sh_s'   Lu -\ff{n(1+\vv)^2} 2  (h_s')^2 uf^{-2},\end{align*}
we derive
\beg{align*} &\d\left(h_s^2 \e^{ 2K_Ds }  \frac{|\nabla u_{t-\tau(s)}|^2}{u_{t-\tau(s)}} (X_{\tau(s)})\right)   \\
& \overset{m}{\ge} \Big( 2(1+\vv) h_sh_s' \e^{2K_Ds} Lu_{t-\tau(s)} -2 \vv h_sh_s'\e^{2K_Ds} \ff{|\nn u_{t-\tau(s)}|^2}{u_{t-\tau(s)}}\Big) (X_{\tau(s)})\, \d s\\
&\quad -\ff{n(1+\vv)^2} 2 \e^{2K_Ds} (h_s')^2( u_{t-\tau(s)}f^{-2})(X_{\tau(s)})\, \d s.
\end{align*}
 Since $\d Lu_{t-s}(X_s) \overset{m}{=} 0 $  so that $\E^x[ L u_{t-{\tau(s)}}(X_{\tau(s)})] =Lu_t(x)$, combining this  with $h_0=1, h_t=0$ and $\tau(0)=0$ implies
\beq \label{GE-1}\beg{split}  & \ff{|\nn u_t|^2}{u_t}(x) \le  \frac{(1+\vv)^2n} 2 \E^x\int_0^t (h_s')^2 \e^{ 2K_Ds}(f^{-2} u_{t-{\tau(s)}})(X_{\tau(s)})\, \d s  \\
&  -  2(1+\vv) (L u_t(x))  \int_0^t  h'_s h_s\e^{2K_Ds}  \, \d s +2\vv \E^x  \int_0^t  h_s'h_s    \e^{  2K_Ds } \frac{|\nabla u_{t-{\tau(s)}}|^2}{u_{t-\tau(s)}} (X_{\tau(s)})\, \d s.\end{split}\end{equation}
   To bound the first term in the right hand side,  we apply It\^o's formula to derive
   \begin{align*}
&\d\big\{ \e^{(2K_D-\bb_{\vv,f})s}f^{-2} u_{t-\tau(s)}\big\}(X_{\tau(s)})\\
&\overset{m}{=} (-\bb_{\vv,f}+2K_D+f^4Lf^{-2}(X_{\tau(s)}))\e^{(2K_D-\bb_{\vv,f})s}u_{t-\tau(s)}(X_{\tau(s)})f^{-2}(X_{\tau(s)})\, \d s\\
&\quad -4 \e^{(2K_D-\bb_{\vv,f})s}f^{-1}(X_{\tau(s)}) \langle \nabla f (X_{\tau(s)}), \nabla u_{t-\tau(s)}(X_{\tau(s)}) \rangle\, \d s\\
& \leq  \e^{(2K_D-\bb_{\vv,f})s} \l(- \bb_{\vv,f}+2K_D+f^4Lf^{-2} +\frac{n(1+\vv)^2}\vv |\nabla f|^2 \r) (f^{-2} u_{t-\tau(s)})(X_{\tau(s)})  \, \d s\\
&\quad +\frac{4\vv}{n(1+\vv)^2}  \e^{(2K_D- \bb_{\vv,f})s} \frac{|\nabla u_{t-\tau(s)}|^2}{u_{t-\tau(s)}}(X_{\tau(s)})\, \d s\\
&\le \frac{4\vv}{n(1+\vv)^2} \e^{(2K_D-\bb_{\vv,f})s} \frac{|\nabla u_{t-\tau(s)}|^2}{u_{t-\tau(s)}}(X_{\tau(s)})\, \d s,
\end{align*} where the last step follows from the  fact that  the definition of $\bb_{\vv,f}$  implies
$$ \bb_{\vv,f}=\sup_{D} \Big\{ 2K_D+f^4Lf^{-2} +\frac{n(1+\vv)^2}\vv |\nabla f|^2  \Big\}.$$
   Hence,  by Gronwall's lemma and $f(X_0)=f(x)=1$, we obtain
\beq\label{S3-1}\beg{split}
&\E^x\l[\e^{2K_Ds}(f^{-2} u_{t-{\tau(s)}})(X_{\tau(s)})\r] \\
&\leq \e^{\bb_{\vv,f} s}u_t(x)+\frac{4\vv \e^{ \bb_{\vv,f}s}}{(1+\vv)^2 n}\int_0^s \e^{(2K_D-\bb_{\vv,f})r} \E^x\l[ \frac{|\nabla u_{t-\tau(r)}  |^2}{u_{t-\tau(r)}} (X_{\tau(r)}) \r]\, \d r.\end{split}
\end{equation}
Since  $h_s'' =- \bb_{\vv,f} h_s'$ due to \eqref{HH}, we obtain
\begin{align*}
\big(h_sh_s'\e^{\bb_{\vv,f} s}\big)'&=\l[(h_s')^2+h_sh''_s +\bb_{\vv,f} h_sh_s'\r]\e^{\bb_{\vv,f} s}=(h_s')^2 \e^{\bb_{\vv,f} s},
\end{align*}
so the
integration by parts formula yields
\begin{align*}
&\int_0^t (h_s')^2 \e^{ \bb_{\vv,f} s}\, \int_0^s \e^{(2K_D-\bb_{\vv,f})r  } \E^x\l[ \frac{|\nabla u_{t-\tau(r)} |^2}{u_{t-\tau(r)}} (X_{\tau(r)})\r]\, \d r\, \d s \\
&=-\int_0^t h_sh_s'\e^{2K_Ds}\E^x\l[ \frac{|\nabla u_{t-\tau(s)}  |^2}{u_{t-\tau(s)} }(X_{\tau(s)}) \r]\, \d s.
\end{align*}
Combining this with \eqref{GE-1} and \eqref{S3-1}, we arrive at
\begin{align*}
 \ff{|\nn u_t|^2}{u_t}(x) \le &\, \frac{n(1+\vv)^2u_t(x)} 2  \int_0^t (h_s')^2 \e^{\bb_{\vv,f}s}  \, \d s
   -  2(1+\vv) (L u_t(x)) \int_0^t   h'_s h_s\e^{2K_Ds}  \, \d s,
\end{align*}
which implies the desired estimate   \eqref{G3} by the definition of $h_s$ in \eqref{HH}.

(b)  Recall that   $\tt X_s:= X_{\tau(s)}$ is the diffusion process  generated by $f^2L$ on $D\setminus \pp D$. Let
$$K_f:= \sup_D \big\{6|\nn f|^2- fLf\big\}.$$
By It\^o's formula, we obtain
$$\d f^{-2}(\tt X_s) \overset{m}{=} f^2 Lf^{-2} (\tt X_s)\d s = f^{-2}(\tt X_s) \big\{6|\nn f|^2- fLf\big\}(\tt X_s) \d s\le K_f f^{-2}(\tt X_s)\d s.$$
This together with $f(\tt X_0)=f(x)=1$ yields
\beq\label{9}  \E^x [f^{-2}(X_{\tau(s)})]\le \e^{K_f s},\ \ s\ge 0. \end{equation} 
For any constant $\bb \ge 0$, let
\beq\label{ELL} \ell_s:= \ff{\e^{-\bb  T(s\land\tau(t))}-\e^{-\bb t}}{1- \e^{-\bb t}},\ \ s\in [0,t].\end{equation}
Then $\ell_s'\le 0, \ell_s=0$ for $s\ge\tau(t)$ where $\tau(t)\le t\land \tau_D$. 
By  \eqref{TO}, \eqref{9} and   the integral transform $s=\tau(r)$, we obtain
\beq\label{N1}\beg{split} &\E^x\bigg[\sup_{s\in [0,t]} \ell_{s }^2\bigg] =\E^x\bigg[\sup_{s\in [0,t]}\bigg(\int_{s}^t \ell_r' \,\d r\bigg)^2\bigg]\\
&\le t\E^x \int_0^t  (\ell_s')^2\d s=\ff{t\bb^2}{(1-\e^{-\bb t})^2} \E^x\int_0^{\tau(t)} \e^{-2\bb T(s\land \tau(t))} f^{-4}(X_s)\,\d s\\
&= \ff{t\bb^2}{(1-\e^{-\bb t})^2} \E^x \int_0^{t} \e^{-2\bb r} f^{-2}(X_{\tau(r)})\,\d r<\infty.  \end{split} \end{equation}
This implies that  condition  \eqref{CV0'} holds, since $\ff{|\nabla u|^2}{u}$ is bounded on $[0,t]\times D$, $K^-= K_D$ on $D$,  and $\si=0$ on $D\cap\pp M$.
 So,  by step (c) in the proof of Theorem \ref{T1},  we obtain  \eqref{P2-0}   for the present $\ell_s$, i.e.
\beq\label{XX}
\beg{split} &\ff{|\nn u_t|^2}{u_t}(x)\le \aa L u_t(x) +\ff{n\aa^2}2 I_1-(\aa-1)I_2,\\
& I_1=\E^x\bigg[u_0(X_t) \int_0^{\tau(t)} \e^{-\ff 2{\aa-1}K_D s}\Big(\ell_s'-\ff{K_D}{\aa-1} \ell_s\Big)^2\,\d s\bigg],\\
& I_2=2\E^x\bigg[\int_0^{\tau(t)}  \e^{-\ff 2{\aa-1}K_D s}\big|\ell_s\ell_s'\big| \ff{|\nabla u_{t-s}|^2}{u_{t-s}}(X_s)\,\d s\bigg]. \end{split}\end{equation}
By \eqref{XW}   and the Markov property,   we have
\beq\label{MM} (L u_{t-\tau(s)})(X_{\tau(s)})=\E^{x} (Lu_0(X_t)|\F_{\tau(s)}),\ \ \ u_{t-{\tau(s)}}(X_{\tau(s)})= \E^{x} (u_0(X_t)|\F_{\tau(s)}).\end{equation}
Moreover, by $f\le 1$,   \eqref{TO}, \eqref{ELL}, $\ell_0=1, \ell_{\tau(t)}=0$ and   the integral transform $s=\tau(r)$, we obtain
\beg{align*}&\int_0^{\tau(t)} \e^{-\ff {2 K_D}{\aa-1} s} \Big(\ell_s'-\ff{K_D}{\aa-1} \ell_s\Big)^2\,\d s\\
&=\int_0^{\tau(t)} \e^{-\ff {2 K_D}{\aa-1} s} (\ell_s')^2\,\d s-2\frac{K_D}{\alpha-1} \int_0^{\tau(t)} \e^{-\ff {2 K_D}{\aa-1} s}  \ell_s'\ell_s\, \d s+ \l(\frac{K_D}{\alpha-1}\r)^2 \int_0^{\tau(t)} \e^{-\ff {2 K_D}{\aa-1} s}   \ell_s^2\, \d s\\
&= \int_0^{\tau(t)} \e^{-\ff {2 K_D}{\aa-1} s} (\ell_s')^2\,\d s +\frac{K_D}{\alpha -1}-  \l(\frac{K_D}{\alpha-1}\r)^2 \int_0^{\tau(t)} \e^{-\ff {2 K_D}{\aa-1} s}   \ell_s^2\, \d s\\
&\le  \int_0^t \ff{ \bb^2 \e^{-2\bb r-\ff {2 K_D}{\aa-1} \tau(r)} }{(1-\e^{-\bb t})^2}     f^{-2}(X_{\tau(r)}) \d r + \ff{K_D}{\aa-1}. \end{align*}
Hence, \eqref{XX} and \eqref{MM} imply 
\begin{align}\label{I1}
 I_1  \le & \ff {\bb^2  }  {(1-\e^{-\bb t})^2}
   \int_0^t \E^x\big[\e^{-2\bb s-\frac{2K_D}{\alpha-1} \tau(s)} (f^{-2} u_{t-\tau(s)} ) (X_{\tau(s)}) \big]\,  \d s  +  \ff{K_D}{\aa-1}  u_t(x). \end{align}
Similarly, by      \eqref{TO}, \eqref{ELL} and   the integral transform $s=\tau(r)$, we obtain
\beq\label{I1'} \beg{split}  I_2&= \ff{2\bb}{(1-\e^{-\bb t})^2} \int_0^{t}  \big(\e^{-2\bb s}-\e^{-\bb(t+s)}\big)   \E^x \l[ \e^{-\frac{2K_D}{\alpha-1}  \tau(s)}  \frac{|\nabla u_{t-\tau(s)}|^2}{u_{t-\tau(s)}} (X_{\tau(s)})\r]\, \d s\\
   &   =  \ff{2\bb^2}{(1-\e^{-\bb t})^2} \int_0^{t} \e^{-\beta s}\l( \int_s^t \e^{-\beta r}\, \d r \r) \E^x \l[ \e^{-\frac{2K_D}{\alpha-1}  \tau(s)}  \frac{|\nabla u_{t-\tau(s)}|^2}{u_{t-\tau(s)}} (X_{\tau(s)})\r]\, \d s\\
   &= \ff{2\bb^2}{(1-\e^{-\bb t})^2} \int_0^{t} \e^{-\bb r}\d r \int_0^r\e^{-\beta s} \E^x \l[ \e^{-\frac{2K_D}{\alpha-1}  \tau(s)}  \frac{|\nabla u_{t-\tau(s)}|^2}{u_{t-\tau(s)}} (X_{\tau(s)})\r]\, \d s. \end{split} \end{equation}
  Since $\tt X_s:= X_{\tau(s)} $ is generated by $f^2L$,    $ u_{t-\tau(s)}(X_{\tau(s)}) $ is a local martingale, and $\tau'(s)=f^2(X_{\tau(s)}),$ by It\^o's formula we obtain
\begin{align*}
&\d \big(\e^{-\beta s}\,\e^{\ff{-2K_D\tau(s)}{\aa-1}}(f^{-2} u_{t-\tau(s)})(X_{\tau(s)})  \big)\\
&\overset{m}{=}\l(-\beta+f^4Lf^{-2}(X_{\tau(s)})-\ff{2K_D}{\aa-1}f^2(X_{\tau(s)})\r) \e^{-\beta s}\,\e^{\ff{-2K_D\tau(s)}{\aa-1}}(f^{-2}u_{t-\tau(s)})(X_{\tau(s)}) \, \d s\\
&\quad - 4\e^{-\beta s}\,\e^{\ff{-2K_D\tau(s)}{\aa-1}}f^{-1}(X_{\tau(s)})\langle\nabla u_{t-\tau(s)}(X_{\tau(s)}), \nabla f(X_{\tau(s)})\rangle\, \d s\\
&\leq \l(-\beta+f^4Lf^{-2}(X_{\tau(s)})+ \frac{2}{\vv}|\nabla f|^2(X_{\tau(s)}) -\ff{2K_D}{\aa-1}f^2(X_{\tau(s)}) \r) 
  \e^{-\beta s}\,\e^{\ff{-2K_D\tau(s)}{\aa-1}}(f^{-2} u_{t-\tau(s)})(X_{\tau(s)})  \, \d s\\
&\qquad +2\vv \e^{-\beta s}\,\e^{\ff{-2K_D\tau(s)}{\aa-1}} \frac{|\nabla u_{t-\tau(s)}|^2}{u_{t-\tau(s)}}(X_{\tau(s)})\, \d s,\ \  \vv>0.
\end{align*}
Choosing $\bb=\tt\bb_{\aa,f}$ and $\vv= \ff{2(\aa-1)}{n\aa^2} $ such that
$$-\beta+f^4Lf^{-2}(X_{\tau(s)})+ \frac{2}{\vv}|\nabla f|^2(X_{\tau(s)}) -\ff{2K_D}{\aa-1}f^2(X_{\tau(s)}\le 0,$$
we derive
$$\d \big(\e^{-\beta s}\,\e^{\ff{-2K_D\tau(s)}{\aa-1}}(f^{-2} u_{t-\tau(s)})(X_{\tau(s)})  \big)\overset{m}{\le} \ff{4(\aa-1)}{n\aa^2} \e^{-\beta s}\,\e^{\ff{-2K_D\tau(s)}{\aa-1}} \frac{|\nabla u_{t-\tau(s)}|^2}{u_{t-\tau(s)}}\, \d s.$$
This together with $f(X_0)=f(x)=1$ and $\tau(0)=0$ yields 
\beg{align*}&  \int_{0}^t \E^x\Big[\e^{-2\bb s-\ff{2K_D\tau(s)}{\aa-1}}(f^{-2} u_{t-\tau(s)})(X_{\tau(s)}) \Big] \,\d s\\
&\le   u_t(x)  \int_0^t \e^{-\bb s}\, \d s  +\ff{4(\aa-1)}{n\aa^2} \int_0^t \e^{-\beta s}\d s\int_0^s  \e^{-\bb r} \E^x\Big[\e^{-\ff{2K_D}{\aa-1}\tau(r)} \ff{|\nn u_{t-\tau(r)}|^2}{u_{t-\tau(r)}}(X_{\tau(r)})\Big]\, \d r.\end{align*}
 Combining this with   \eqref{I1}, \eqref{I1'} and
 $\vv= \ff{2(\aa-1)}{n\aa^2},$ we derive
 $$  \ff{n\aa^2}{2}I_1-(\aa-1)I_2 \le \ff{n\aa^2}2 \l(\ff{K_D}{\aa-1}  +\frac{\beta}{1-\e^{-\beta t}}\r) u_t(x),$$
 where $\bb=\tt\bb_{\aa,f}$. Then \eqref{D4} follows from \eqref{XX}.

\end{proof}

To derive explicit estimates from Theorem \ref{T2}, we take   $D=B(x,R)$, the geodesic ball in $M$ centered at $x$ with radius $R$, for any $R>0$ and $x\in M$.  Let
$$K_{x,R}:=\inf \Big\{\Ric_Z^{(n-m)}(v,v):\   v\in \cup_{y\in B(x,R)} T_y M,\ |v|=1\Big\}.   $$   We have the following result.

\begin{cor}\label{C2} Assume that $\pp M$ is either convex or empty. Let $x\in M$ and $t>0$.    Then the following assertions hold.
 \beg{enumerate}
 \item [$(1)$]  For any $\vv\in (0,1)$, $\eqref{G3}$ holds for $\beta_{\vv, f}$ replaced by
 \beq\label{GH} \bb_{\vv,R}=2 K_{x,R} +\ff{\pi}{2R} \ss{K_{x,R}(n-1)} +\ff{\pi^2}{4R^2}\Big[4+   \big(\vv^{-1}(1+\vv)^2+2\big)n \Big].\end{equation}
  \item[$(2)$]
 For any $\alpha>1$, $\eqref{D4}$ holds with  $\tt\bb_{\aa, f}$ replaced by
 \beq\label{GH'} \tt\bb_{\aa,R}= \ff{\pi^2}{2 R^2}\Big(2+n+\ff{n\aa^2}{2(\aa-1)}\Big)+ \ff{\pi}{2R}\ss{K_{x,R}(n-1)}+\ff{2K_{x,R}}{\aa-1}.\end{equation}

\end{enumerate}
\end{cor}

\begin{proof}
Let $D=B(x,R)$, and let $\rr_x$ be the Riemannian distance to point $x$. Choose
\beq\label{*01}f= \cos\Big(\ff{\pi\rr_x}{2R}\Big).\end{equation}
Since $\pp M$ is either convex or empty, we have $N\rr_x|_{\pp M}\le 0$ so that $f$ satisfies \eqref{G0}.

Next, by   the curvature-dimension condition \eqref{C} on $D$ for $K=-K_{x,R}$, and taking
$$\varphi(s):=\ff{\sinh(\ss{K_{x,R}/(n-1)}\,s)}{\sinh(\ss{K_{x,R}/(n-1)}\,\rho_x)},\ \ s\in [0,\,\rho_x],$$
 in \cite[(3)]{Q97}, we obtain the following Laplacian comparison theorem:
$$L\rr_x(y)\le \ss{K_{x,R}(n-1)} \coth\Big(\ss{K_{x,R}/(n-1)}\,\rr_x(y)\Big),\ \ x\ne y\in D\setminus {\rm cut}(x),$$
where ${\rm cut}(x)$ is the cut-locus of $x$, such that $f$ is smooth on $D\setminus {\rm cut}(x).$

(a) When ${\rm cut}(x)=\emptyset,$  $f\in C_b^2(D)$ satisfying \eqref{G0}. Since $\pp M$ is convex or empty, we have $\si=0$.   Combining these with $|\nn\rr_x|=1$,  we conclude that  $\bb_{\vv,f}$ in  Theorem \ref{T2}(1) can be estimated:
  \beg{align*} &\bb_{\vv,f}:= \sup_D\bigg\{2K_{x,R} - 2 f Lf +\Big(6+ \ff{(1+\vv)^2n}\vv\Big)|\nn f|^2\bigg\}\\
&\le \sup_{r\in [0,R]} \bigg\{2K_{x,R} + \ff{\pi}{2R} \sin\Big(\ff{\pi r} R\Big) \ss{K_{x,R}(n-1)}\coth\Big(\ss{K_{x,R}/(n-1)}\, r\Big)\\
&\qquad\qquad\qquad + \ff{\pi^2}{2R^2} \cos^2\Big(\ff{\pi r}{2R}\Big)+\Big(6+\ff{(1+\vv)^2n }\vv \ff{\pi^2}{4R^2}\Big)\sin^2\Big(\ff{\pi r}{2R}\Big)\bigg\}.\end{align*}
Noting that $\sin r\le 1\land r,\ \sin^2r+\cos^2 r=1$ and $ \coth r\le 1+r^{-1}$ for $r>0$, this implies
\beg{align*}  \bb_{\vv,f}&\le 2K_{x,R} +\ff{\pi}{2R} \Big(\ss{K_{x,R}(n-1)} +\ff{\pi(n-1)}R\Big) +\Big(6+\ff{n(1+\vv)^2}\vv\Big)\ff{\pi^2}{4R^2}\\
&= 2 K_{x,R} +\ff{\pi}{2R} \ss{K_{x,R}(n-1)} +\ff{\pi^2}{4R^2}\Big[4+   \big(\vv^{-1}(1+\vv)^2+2\big)n \Big].\end{align*} 
So, by Theorem \ref{T2}(1), \eqref{G3} holds for $\bb_{\vv,f}$ replaced by $\beta_{\vv, R}$ in \eqref{GH}.

Similarly,
\beg{align*} &\tt\bb_{\aa,f}:=  \sup_D\bigg\{\Big(6+\ff{n\aa^2}{\aa-1} \Big) |\nn f|^2 -2fLf -\ff{2K_{x,R}}{\aa-1} f^2\bigg\}\\
&\le  \Big(6+\ff{n\aa^2}{\aa-1} \Big)\ff{\pi^2}{4R^2} +\sup_{r\in [0,R]}\l\{ \ff{\pi}{2R} \sin\Big(\ff{\pi r} R\Big) \ss{K_{x,R}(n-1)}\coth\Big(\ss{K_{x,R}/(n-1)}\, r\Big)\r\}+\ff{2K_{x,R}}{\aa-1}\\
&\le \ff{\pi^2}{2 R^2}\Big(2+n+\ff{n\aa^2}{2(\aa-1)}\Big)+ \ff{\pi}{2R}\ss{K_{x,R}(n-1)}+\ff{2K_{x,R}}{\aa-1}.\end{align*}
So,  by Theorem \ref{T2}(2), \eqref{D4} holds for $\tt\bb_{\aa,f}$ replaced by $\tt\bb_{\aa,R}$ in \eqref{GH'}.

(b) When cut$(x)\ne\emptyset,$ noting that $N\rr_x|_{\pp M}\le 0$ by the convexity of  $\pp M$,
the It\^o  formula for $\rr_x(X_t)$ due to \cite{K} implies
\begin{align*}  \d \rr_x(X_t)&\le L \rr_x(X_t) \d t +\ss{2}\,\d b_t\\
&\le \ss{K_{x,R}(n-1)} \coth\Big(\ss{K_{x,R}/(n-1)}\,\rr_x(X_t)\Big)\d t+\ss 2 \,\d b_t,\ \ t\le \tau_D,\end{align*}
where $b_t$ is the one-dimensional Brownian motion. With this inequality and the fact that $\cos r$ is smooth and decreasing in $r\in [0,\ff \pi 2]$,
 the argument in the proof of Theorem \ref{T2} still works for the present choice of $f$, so that the proof
 is finished as in the above step (a).

\end{proof}


\end{document}